\documentclass[12pt]{article}
\usepackage{amsfonts,amssymb,amsmath,color}
\usepackage[english]{babel}
\usepackage{a4wide}
\usepackage{amssymb}
\usepackage{amsthm}
\usepackage{pdfsync}
\usepackage{hyperref}

\usepackage{numbysec}

-\marginparwidth \oddsidemargin -4mm \evensidemargin \oddsidemargin

\newcommand{\eps}{\varepsilon}
\newcommand{\commentout}[1]{}

\newcommand{\ff}{\mathfrak f}

\newcommand{\bbE}{\mathbb{E}}
\newcommand{\bbP}{\mathbb P}

\newcommand{\bbV}{V}

\newcommand{\la}{\lambda}

\newcommand{\jed}{\textbf{1}}
\newcommand{\jj}{\mathbf{j}}
\newcommand{\ba}{\mathbf{a}}
\newcommand{\bb}{\mathbf{b}}
\newcommand{\bA}{\mathbf{A}}

\newcommand{\fa}{\mathfrak{a}}

\newcommand{\aipq}{a_\jj}

\newcommand{\bipq}{b_\jj}
\newcommand{\baraipq}{a_\jj}
\newcommand{\barbipq}{b_\jj}

\newcommand{\partchia}{\chi^{(\ep)}_{q,a_i}}
\newcommand{\partchib}{\chi^{(\ep)}_{q,b_i}}

\newcommand{\tpartchiaB}{\tilde\chi_{q,a_i}}
\newcommand{\tpartchibB}{\tilde\chi_{q,b_i}}

\newcommand{\lambab}{\fa}
\newcommand{\cZ}{\mathcal{Z}}
\newcommand{\Xep}{X^{(\ep)}}

\newcommand{\Vep}{\bbV^{(\ep)}}

\newcommand{\tep}{\frac{t}{\ep^2}}
\newcommand{\bbW}{W}

\newcommand{\bbU}{U}

\newtheorem{twr}[theorem]{Theorem}
\newtheorem{twrg}{Theorem$\vphantom{1}^\star$}[section]
\newtheorem{lem}[theorem]{Lemma}
\newtheorem{Propozycja}[theorem]{Proposition}
\newtheorem{corollary}[theorem]{Corollary}

\newtheorem{definicja}[theorem]{Definition}

\newtheorem{remark}[theorem]{Remark}

\newcommand{\B}{{\@Bbb B}}
\newcommand{\C}{{\@Bbb C}}

\newcommand{\F}{{\@Bbb F}}
\renewcommand{\P}{{\@Bbb P}}

\newcommand{\Q}{{\@Bbb Q}}
\newcommand{\bQ}{{\@Bbb Q}}
\newcommand{\N}{{\@Bbb N}}
\newcommand{\R}{{\mathbb R}}

\newcommand{\W}{{\@Bbb W}}

\newcommand{\al}{\alpha}
\newcommand{\si}{\sigma}

\newcommand{\Om}{\Omega}

\newcommand{\ep}{\varepsilon}

\newcommand{\cA}{\@s A}
\newcommand{\cC}{\@s C}
\newcommand{\cD}{\@s D}
\newcommand{\cF}{\@s F}
\newcommand{\cG}{\@s G}
\newcommand{\cI}{\@s I}
\newcommand{\cJ}{\@s J}
\newcommand{\cK}{\@s K}
\newcommand{\cL}{\@s L}
\newcommand{\cN}{\@s N}
\newcommand{\cM}{\@s M}
\newcommand{\cO}{\@s O}
\newcommand{\cP}{\@s P}
\newcommand{\cR}{\@s R}
\newcommand{\cS}{\@s S}
\newcommand{\cT}{\@s T}
\newcommand{\cW}{\@s W}
\newcommand{\cX}{\@s X}
\newcommand{\cY}{\@s Y}
\newcommand{\sigipq}{\sigma_\jj}

\newcommand{\bma}{\@bm a}
\newcommand{\bmb}{\@bm b}
\newcommand{\bmc}{\@bm c}
\newcommand{\bmd}{\@bm d}

\newcommand{\bme}{\@bm e}
\newcommand{\bmf}{\@bm f}
\newcommand{\bmg}{\@bm g}
\newcommand{\bmh}{\@bm h}
\newcommand{\bmi}{\@bm i}
\newcommand{\bmj}{\@bm j}
\newcommand{\bmk}{\@bm k}
\newcommand{\bml}{\@bm l}
\newcommand{\bmm}{\@bm m}
\newcommand{\bmn}{\@bm n}
\newcommand{\bmo}{\@bm o}
\newcommand{\bmp}{\@bm p}
\newcommand{\bmq}{\@bm q}
\newcommand{\bmr}{\@bm r}
\newcommand{\bms}{\@bm s}
\newcommand{\bmt}{\@bm t}
\newcommand{\bmu}{\@bm u}
\newcommand{\bmw}{\@bm w}
\newcommand{\bmv}{\@bm v}
\newcommand{\bmx}{\@bm x}
\newcommand{\bx}{\@bm x}
\newcommand{\bmy}{\@bm y}
\newcommand{\bmz}{\@bm z}
\newcommand{\bbN}{\mathbb N}
\newcommand{\by}{\@bm y}
\newcommand{\bmzero}{\@bm 0}

\newcommand{\ga}{\gamma}

\newcommand{\gA}{\@g A}
\newcommand{\gD}{\@g D}
\newcommand{\gJ}{\@g J}
\newcommand{\gF}{\@g F}
\newcommand{\gM}{\@g M}
\newcommand{\gR}{\@g R}
\newcommand{\bbR}{\R}

\newcommand{\mbR}{\mathbb{R}}
\newcommand{\bbC}{\mathbb{C}}

\begin{document}
\author{Tomasz Komorowski\thanks{Polish Academy of Science, ul. \'{S}niadeckich 8
00-656 Warszawa}\\Tymoteusz Chojecki\thanks{Institute of Mathematics,  UMCS,
pl. Marii Curie-Sk\l odowskiej 1, 20-031, Lublin}}

\title{Homogenization of an advection equation with locally stationary
 random coefficients}
\maketitle
\begin{abstract}
In the  paper we consider the solution of an advection equation with rapidly changing coefficients
$\partial_t u_\eps+(1/\eps)V(t/\eps^{2},x/{\eps})\cdot\nabla_x
u_\eps=0$ for $t<T$ and  $u_\eps(T,x)=u_0(x)$, $x\in\bbR^d$. Here $\eps>0$ is some small parameter and the drift term $\left(V(t,x)\right)_{(t,x)\in \bbR^{1+d}}$ is assumed to be  a $d$-dimensional, vector valued random field with incompressible spatial realizations.
We prove that when the field is  Gaussian, locally stationary, quasi-periodic in the $x$ variable and strongly mixing in time  the solutions $u_\eps(t,x)$ converge in law, as $\eps\to0$, to $  u_0(x(T;t,x))$, where  $\left(x(s;t,x)\right)_{s\ge t}$ is a diffusion satisfying $x(t;t,x)=x$. The averages of $u_\eps(T,x)$ converge then to the solution of the corresponding Kolmogorov backward equation.
\end{abstract}
\numberbysection

\section{Introduction}

In the present paper, we consider solutions of linear advection equations with rapidly oscillating
random coefficients
of the form
\begin{equation}\label{e.maineq}
\begin{aligned}
&\partial_t u_\eps(t,x)+\frac{1}{\eps}V\Big(\frac{t}{\eps^2},\frac{x}{\eps}\Big)\cdot\nabla_x
u_\eps(t,x)=0,  \\
&u_\eps(T,x)=u_0(x), \quad  t<T, x\in\R^d.
\end{aligned}
\end{equation}
Here, $\left(V(t,x)\right)_{(t,x)\in\bbR^{1+d}}$ is a random, zero-mean, incompressible, Gaussian,
vector-valued random field and $\eps>0$. We are interested in the diffusive scaling limit of
the solutions, as the parameter $\eps$ tends to $0$. Equation
\eqref{e.maineq} appears e.g. in  the {\em passive scalar} model  that
  describes a concentration of particles drifting in a time-dependent, incompressible
  random  flow
  and has applications in both turbulent diffusion and stochastic homogenization, see e.g. \cite{krama,kraichnan,warhaft,sreenivasan} and the references therein.
The model has been extensively studied, both in the mathematics and physics  literature, under various assumptions
on the advection term $V(t,x)$. 
A typical result states that, if the field is stationary and sufficiently strongly mixing, then   the underlying random characteristics  (that correspond to the trajectory realizations of the drifting particle) converge in law to
a zero mean Brownian motion $(\beta_t)_{t\ge0}$ whose  covariance matrix $[a_{p,q}]_{p,q=1,\ldots,d}$ is determined by the statistics of $V(\cdot,\cdot)$, see e.g. \cite{Prof,bp,cf,carmona-xu,fk,kola,kom}. 
In that case the laws of the solutions of \eqref{e.maineq}
 converge, as $\eps\to0$, to $u_0(x+\beta_{T-t})$. Its expectation
$\bar u(t,x)$  satisfies
\begin{equation}
\label{bheat-11}
\begin{aligned}
&\partial_t \bar u(t,x)+\frac12\sum_{p,q=1}^d a_{pq}\partial_{x_p,x_q}^2\bar
u(t,x)=0,\quad t\le T,
\\
&\bar u(T,x)=u_0(x).
\end{aligned}
\end{equation}
Since the coefficients of  equation \eqref{bheat-11} do not depend on the spatial variable, the limiting procedure is sometimes  referred to as homogenization.
Stationarity and ergodicity of the velocity field play a   crucial role in substantiating the existence of the limit in homogenization, as  the argument relies
 on an application of some form of an ergodic
theorem.
%

The main purpose of the present article is to investigate the
situation when the coefficients of the advection equation
\eqref{e.maineq} are no longer  stationary. We  assume instead that  the velocity
can be written as $V(t,x,\eps x)$, for some random vector field $V(t,x,y)$, where
for a fixed $y$ the field
 is assumed to be stationary and ergodic in the variables $(t,x)$. 
  The variable $y$ represents a 'slow' parameter i.e. when
 $\ep\ll 1$ then the statistics of the field $V(t,x,\eps x)$
 suffer a significant change only when
$|x|\sim1/\ep$.
For technical reasons we shall also assume that $V(t,x,y)$ is quasi-periodic in the $x$ variable. A more precise description of the fields considered in
the paper
is given in Section \ref{mod1}.
In our  main result,  see Theorem \ref{glowne} below, we show that, as
 $\ep\to0$,  the limit of $u_\eps(t,x)$, in the law, is given by $u_0(x(T;t,x))$, where  $\left(x(t;s,x)\right)_{t\ge s}$ is the diffusion, starting at $s$ at position $x$ with the generator given by the differential operator defined in \eqref{gen-diff}. Then  $\bar u(t,x)$ - the  expectation  of $u_0(x(T;t,x))$ - is the
 solution of the respective  Kolmogorov backward parabolic equation 
\begin{equation}
\label{bheat-11a}
\begin{aligned}
&\partial_t \bar u(t,x)+\sum_{p=1}^d B_{p}(x)\partial_{x_p}\bar
u(t,x)+\frac12\sum_{p,q=1}^d A_{pq}(x)\partial_{x_p,x_q}^2\bar
u(t,x)=0,\quad t< T,
\\
&\bar u(T,x)=u_0(x)
\end{aligned}
\end{equation}
with the respective coefficients appearing in the definition of the generator.

Homogenization  of parabolic  and elliptic equations with locally periodic coefficients has been considered in
Chapter 6 of the
book \cite{BPL}. The generalization to the case of random  parabolic
equations in divergence
form with locally stationary and ergodic coefficients,  has been done in
\cite{Rodes}. An anologous question in  the case of difference
equations in  divergence form in dimension one
has been considered in
\cite{Olla-Siri}. The notion of local ergodicity used in ibid. differs
from the one in \cite{Rodes} and is conceptually closer to the one considered in the present paper.
Homogenization
 of linear
parabolic equations  in  non-divergence form with non-stationary
coefficients has been treated in \cite{boko}. Somewhat related problem of averaging with two scale (fast and slow) motion,
but under a scaling different from ours,
 has been
also considered in the literature, see e.g. \cite{Kifer,Kelly} and
references therein. 

 Our proof is based on 
an analysis of the asymptotics of the random characteristics corresponding to the advection equation  \eqref{e.maineq}.
We apply  the corrector method to eliminate the large amplitude terms that arise in the description of the characteristics. This requires showing  regularity
of the correctors with respect to the parameter that corresponds to the slow variable of the velocity field.
{In Section \ref{pochodna} we prove several results
  concerning the regularity properties of the corrector, which seem to
  be of independent interest. They are obtained by a technique
based on an application of the Malliavin calculus, which is related to the method used in \cite{hm}
to establish asymptotic strong Feller property for the solutions of
stochastic Navier-Stokes equations in two dimensions. It is essentially the only place in our argument that requires the hypothesis of quasi-periodicity of the flow.}  To show the existence of the limit (in law) of the processes corresponding to the random characteristics we apply an averaging lemma, see Lemma \ref{ergodyczne} below, which is a version of a suitable ergodic theorem.

{ The organization of the paper is as follows. In Section
\ref{piaty} we present more detailed description of the
model, which we are going to study and formulate some of its basic properties. The main result, see   Theorem  \ref{glowne} below, is formulated in Section  \ref{ProcLagran}.  Its proof is contained in Section \ref{sec5}. For the notational convenience we conduct the argument only for dimension $d=2$. 
Section \ref{sec4} contains a detailed description of the two dimensional case. It is clear from our proof that it can be easily generalized to the  case of an arbitrary dimension. However this can be done at the expense of a considerably heavier notation, see Section \ref{sec5.5} for the discussion of the general dimension situation.  Finally, Sections \ref{pochodna} -- \ref{sec8} are devoted to showing  some technical results needed for the proof of our main theorem. 

\subsubsection*{Acknowledgment}
This work was partially supported by the grant 346300 for IMPAN from the Simons Foundation and the matching 2015-2019 Polish MNiSW fund.
Both authors acknowledge the support of the National Science Centre:
NCN grant 2016/23/B/ST1/00492. {T.K. thanks prof. S. Peszat for
stimulating discussions
on the subject of the paper.}

\section{Preliminaries}

\label{piaty} 


\subsection{Quasi-periodic, locally stationary fields of coefficients}\label{mod1}


Given $\eps>0$ we let $V_\eps=(V_{1,\eps},\ldots,V_{d,\eps}):\bbR^{1+d}\times \Om\to\bbR^d$ be a random,  vector field. Here
$(\Omega,\mathcal{F},\bbP)$ is a probability space.  We let $\bbE$ be the expectation with respect to $\bbP$. To ensure  that
the field has divergence free realizations we let
\begin{equation}
\label{050705}
V_{m,\eps}(t,x)=\sum_{l=1}^d\partial_{x_l}H_{l,m}^{\eps}(t,x),\quad m=1,\ldots,d,
\end{equation}
where $H^\eps(t,x):=[H_{l,m}^\eps(t,x)]_{l,m=1,\ldots,d}$ is a $d\times d$ anti-symmetric matrix
valued random, quasi-periodic field of the form
$H_{l,m}^\eps(t,x)=H_{l,m}(t,x,\eps x)$, with
$$
H_{l,m}(t,x,y)=\sum_{i=1}^N\left[a_{\jj}(t,y)\cos(k_i\cdot
  x)+b_{\jj}(t,y)\sin(k_i\cdot x)\right],\quad (t,x,y)\in\R^{1+2d},
$$ 
where $N$ is fixed natural number and  $\jj$ denotes the multi-index
$(i,l,m)$ made of three components $i=1,\ldots,N$ and $l,m=1,\ldots,d$.
Here we let also  $k_i=(k_{i,1},\ldots,k_{i,d})\in\R^d$.

The random fields
$\left(a_{\jj}(t,y)\right)_{(t,y)\in\bbR^{1+d}}$,
$\left(b_{\jj}(t,y)\right)_{(t,y)\in\bbR^{1+d}}$ for
\begin{equation}
\label{Z}
\jj\in  Z:=\left\{(i,l,m):\ i=1,\dots,N,1\leq l<m\leq d\right\}
\end{equation}
are of the form
\begin{align}\label{roz}\begin{aligned}
&a_\jj(t;y)=\sqrt{2\alpha_\jj(y)}\sigma_\jj(y)\int_{-\infty}^te^{-\alpha_\jj(y)(t-s)}dw_{\jj,a}(s),\\
&b_\jj(t;y)=\sqrt{2\alpha_\jj(y)}\sigma_\jj(y)\int_{-\infty}^te^{-\alpha_\jj(y)(t-s)}dw_{\jj,b}(s).
\end{aligned}\end{align}
Here $w_{\jj,a}(t),w_{\jj,b}(t),\ \jj\in Z$ are  independent, two-sided
 one dimensional standard Brownian motions.
For the indices $\jj=(i,l,m)$, with  $m\ge l$   we  let 
$$
a_{i,l,m}(t;y)=-a_{i,m,l}(t;y),\quad
b_{i,l,m}(t;y)=-b_{i,m,l}(t;y),\quad\,i=1,\ldots,N.
$$
The above implies in particular that
$$
a_{i,m,m}(t;y)=b_{i,m,m}(t;y)\equiv0,\quad m=1,\ldots,d, \,i=1,\ldots,N.
$$
Functions
$\alpha_\jj(\cdot)$, $\sigma_\jj(\cdot)$ are assumed  to belong to
$C^2_b(\mbR^d)$ - the class of twice, continuously differentiable
functions with bounded derivatives and satisfy
\begin{align}\label{gamma}\begin{aligned}
&1/\sigma_*\geq\sigma_{\jj}(y)\geq \sigma_*,\quad
1/\gamma_0\geq\alpha_{\jj}(y)\geq \gamma_0\quad \mbox{ for }y\in\bbR^d
\end{aligned}\end{align}
and some
$\gamma_0,\sigma_*\in(0,1)$.
It is clear from  \eqref{roz} that for each $y$ the processes $\left(a_{\jj}(t,y)\right)_{t\in\bbR}$,
$\left(b_{\jj}(t,y)\right)_{(t,y)\in\bbR}$ are the stationary solutions of the It\^o stochastic differential equations
\begin{equation}\label{Ornstein}
\begin{aligned}
&da_\jj(t;y)=-\alpha_\jj(y)a_\jj(t;y)dt+\sqrt{2\alpha_\jj(y)}\sigma_\jj(y)dw_{\jj,a}(t),\\
&db_\jj(t;y)=-\alpha_\jj(y)b_\jj(t;y)dt+\sqrt{2\alpha_\jj(y)}\sigma_\jj(y)dw_{\jj,b}(t).
\end{aligned}
\end{equation}

\subsection{Markov property of the process }

 The
generator $\mathbb L^y$ of the $\R^{2S}$-valued process ${\frak
  a}(t,y):=\left(a_\jj(t;y),b_\jj(t;y)\right)_{\jj\in Z}$ equals 
 \eqref{Ornstein}
\begin{equation}\label{genOrnstein}
\mathbb L^yF(\lambab)=\sum_{\jj\in
Z}\left(L_{\jj, a_\jj}^y+L_{\jj,b_\jj}^y\right)F(\lambab),\quad F\in C^2(\R^{2S}),\,\fa :=\left(a_\jj,b_\jj\right)_{\jj\in Z}\in\bbR^{2S}.
\end{equation}
Here $S$
denotes the  cardinality of $Z$.
The one dimensional differential operators $L_{\jj,a_\jj}^y$,
$L_{\jj,b_\jj}^y$ act on the $a_\jj$ and $b_\jj$ variables, respectively, with
\begin{align}\label{generatorL}
&L_{\jj,a}^yf(a):=\alpha_\jj(y)\left[\si_\jj^2(y)f''(a)-
a f'(a)\right],\quad  f\in C^2(\R).
\end{align}
The Gaussian product measure 
\begin{equation}\label{miara}
\nu_{*}^y(d\fa)=\prod_{\jj\in
Z}\Phi_{\sigma_\jj(y)}(a_\jj)\Phi_{\sigma_\jj(y)}(b_\jj)d\aipq
d\bipq,
\end{equation}
where for $\sigma>0$
$$
\Phi_{\sigma}(a):=\frac{1}{\sigma\sqrt{2\pi}}\exp\left\{-\frac{a^2}{2\sigma^2}\right\},\quad
a\in\mbR,
$$
is invariant under the dynamics corresponding to the generator $\mathbb L^y$,
i.e.
\begin{equation}
\label{inv}
\int_{\R^{2S}}\mathbb L^yFd\nu_*^y=0
\end{equation}
for any $F\in C^2(\R^{2S})$ of at most polynomial growth.

The following result  is a consequence of 
 Propositions 12.4
and  12.14, part v) of
 \cite{Ksiazka}.
\begin{Propozycja}[Spectral gap property of the generator] Fix $y\in\mbR^d$. The   set ${\cal P}$ of
 polynomials on $\mbR^{2S}$ constitutes a core of $\mathbb L^y$.
 Assume that $F\in {\cal P}$ satisfies
\begin{equation}\label{calka}
\int_{\mbR^{2S}} Fd\nu_{*}^y=0.
\end{equation}
Then,
\begin{equation}\label{spectralgapOU2}
-\langle \mathbb L^y F,F\rangle_{\nu_{*}^y}\geq
\gamma_0||F||_{L^2(\nu_{*}^y)}^2,
\end{equation}
where $\gamma_0$ was introduced in $\eqref{gamma}$.
\end{Propozycja}

\subsection{Homogeneous fields}

For $x\in\bbR^d$ and  ${\frak
  a}=\left((a_\jj,b_\jj)\right)_{\jj\in Z}$ we let $\tau_x:\bbR^{2S}\to\bbR^{2S}$ by the formula
$ \tau_x( {\frak a})=\left(a_\jj',b_\jj'\right)_{\jj\in Z}$, where
\begin{align}\begin{aligned}\label{pz} &
a'_\jj:=a_\jj\cos(k_i\cdot x)+b_\jj\sin(k_i\cdot x),\\
&  b'_\jj:=-a_\jj\sin(k_i\cdot x)+b_\jj\cos(k_i\cdot x),\quad \jj\in Z.
\end{aligned}\end{align}
  It is easy to check that $\left(\tau_x\right)_{x\in\bbR^d}$ forms a
  group of transformations with $\tau_x\tau_y=\tau_{x+y}$, $x,y\in\bbR^d$.
For the function $G:\mbR^{2N}\times\mbR^{d}\to\mbR$ we denote by $\tilde G:\bbR\times\mbR^{2d}\times\Omega\to\bbR$ the random field 
given by
\begin{equation}\label{tild}
\tilde G(t,x,y):=G(\tau_x\left( {\frak a}(t;y)\right),y),
\end{equation}
where $\left({\frak a}(t;y)\right)$ is the process given by \eqref{roz}.

\section{Statement of the main result}

\label{ProcLagran}

Our main result concerns the diffusive scaling limit for the random characteristics of \eqref{e.maineq}. They are given by the trajectories of solutions of the ordinary
differential equation with the random right hand side given by the field
$V_\eps(t,x)$, defined in Section \ref{mod1}. More precisely, suppose that
 \begin{equation}\label{dynam22}
\left\{\begin{array}{ll}
&\dfrac{dX_\ep^{s,x_0}(t)}{dt}=V_\eps\left(t,X_\eps^{s,x_0}(t)\right),\\
&\\
&X_\eps^{s,x_0}(s)=x_0,\end{array} \right.
\end{equation}
where $\ep>0$ and $x_0\in\bbR^d$ and $s,t\in\bbR$.
The  diffusively scaled processes
$x_\eps(t;s,x_0):=\ep X_\eps^{s/\eps^2,x_0/\eps}\left(t/\ep^2\right)$ 
satisfy
\begin{equation}\label{poleweup}
\left\{\begin{array}{ll}
&\dfrac{d x_\eps(t;s,x_0) }{dt}=\dfrac{1}{\ep}W\left(\dfrac{t}{\ep^2},\dfrac{x_\eps(t;s,x_0) }{\ep}, x_\eps(t;s,x_0)\right)+U\left(\dfrac{t}{\ep^2},\dfrac{x_\eps(t;s,x_0)}{\ep},x_\eps(t;s,x_0)\right),\\
&x_\eps(s;s,x_0)=x_0,\end{array}\right.
\end{equation}
where $W=(W_1,\ldots,W_d)$, $U=(U_1,\ldots,U_d)$ are given by
$$
W_m(t,x,y)=\sum_{l=1}^d\partial_{x_l}H_{l,m}(t,x,y),\quad U_m(t,x,y)=\sum_{l=1}^d\partial_{y_l}H_{l,m}(t,x,y).
$$
The main result of the present paper can be formulated as follows.
\begin{twr}\label{glowne}
For a given $(s,x_0)\in\bbR^{1+d}$ the processes $\left(x_\eps(t;s,x_0)\right)_{t\geq s}$ converge in law over
$C\left([s,\infty);\mbR^d\right)$, when $\ep\to0$, to the diffusion $\left(x(t;s,x_0)\right)_{t\geq s}$,
which starts at  time $s$ at $x_0$ and whose generator is given by
\begin{align}
\label{gen-diff}
\overline{\mathcal{L}}f(y)=\sum_{l=1}^dB_l(y)\partial_{y_l}f(y)+\frac12\sum_{l,l'=1}^dA_{l,l'}(y)\partial^2_{y_{l,l'}}f(y),
\quad f\in C^2(\mbR^{d}),\ y\in\mbR^d,
\end{align}
where coefficients $B_l(\cdot)$ and $A_{l,l'}(\cdot),\
l,l'=1,\dots,d$ are defined by formulas $\eqref{wspD}$ below.
\end{twr}

Using the characteristics of \eqref{e.maineq} we can write the
solution in the form
\begin{equation}
\label{extra-2}
u_\eps(t,x)=u_0\left(x_\eps(T;t,x)\right),\quad t\le T,\,x\in\bbR^d.
\end{equation}
As an immediate  corollary of Theorem \ref{glowne} we conclude the
following.
\begin{corollary}\label{glowne1}
Suppose that $u_\eps(t,x)$ is the solution of \eqref{e.maineq} with
$u_0$ that is bounded and continuous. Then, the random variables
$ u_\eps(t,x)$ converge in law, as $\eps\to0$, to 
$u_0(x(T;t,x))$. In particular, $\bar u(t,x):=\lim_{\eps\to0}\bbE u_\eps(x(T;t,x))$  is the
bounded solution of 
\begin{align}
\label{extra-1}
&\partial_t \bar
  u(t,x)+\sum_{l=1}^dB_l(x)\partial_{x_l}\bar u
  (t,x)+\frac12\sum_{l,l'=1}^dA_{l,l'}(x)\partial^2_{x_{l,l'}}\bar u(t,x)=0,\quad t<T,\,x\in\mbR^d,\nonumber\\
  &
  \\
&\bar u(T,x)=u_0(x),\,\nonumber
\end{align}
where coefficients $B_l(\cdot)$ and $A_{l,l'}(\cdot),\
l,l'=1,\dots,d$ are as in Theorem $\ref{glowne}$.
\end{corollary}

\begin{remark}
{{\em 
Note that the random variables  $u_\eps(t,x)$, given by \eqref{extra-2}, do not  become deterministic in the limit, as $\eps\to0$.
According to Corollary  \ref{glowne1},  they converge in law to
$u_0(x(T;t,x))$. The diffusion $x(T;t,x)$ is not  deterministic   as its diffusivity matrix $[A_{l,l'}(x)]$, given by \eqref{wspD} below, is non-degenerate. Indeed, if otherwise one would conclude easily from  \eqref{wspD} that the gradient of the respective corrector field, defined in \eqref{011702j}, would vanish.  This, in turn would contradict
the fact that the right hand side of \eqref{011702j}  is non-trivial. }}

{\em The present case should be contrasted with the situation when $u_\eps(t,x)$ satisfies an advection-diffusion equation 
\begin{equation}
\label{advec-11a}
\begin{aligned}
&\partial_t u_\eps(t,x)+\dfrac1\eps V\left(\dfrac{t}{\eps^2},\dfrac{x}{\eps}\right)\cdot\nabla_x u_\eps(t,x)+\kappa\Delta_xu_\eps(t,x)=0,\\
&u_\eps(T,x)=u_0(x),
\end{aligned}
\end{equation}
with a non-zero ellipticity constant $\kappa>0$. This case has been considered  in \cite{Rodes}, under somewhat differently formulated property of local stationarity of the random drift. 
One can show that then for any $t\le T$ the random variables   
$u_\eps(t,x)$
converge, in probability, to a deterministic limit $\bar u(t,x)$ given by    the solution of \eqref{extra-1} with appropriate coefficients.
Contrary to the present case of advection equation \eqref{e.maineq}, the diffusion term in (\ref{advec-11a}) provides enough extra averaging so that the limit becomes deterministic.}
\end{remark}




\section{Prelude to the proof of Theorem \ref{glowne}}

\label{sec4}

\subsection{Two dimensional case} To simplify the notation we shall assume that $(s,x_0)=(0,0)$. In that case we shall write 
$x_\eps(t):=x_\eps(t;0,0)$. To further lighten up the notation we shall present the
argument for the case $d=2$. Then, the antisymmetric matrix $H(t,x,y)$
can be described by a scalar field and, as a result, this allows to reduce
the multi-index $\jj$ to just a scalar index $i\in\{1,\ldots,N\}$. The
case of an arbitrary dimension $d$ requires the same consideration,
however the argument will be obscured by some heavy notation.

In this case velocity field $V_\eps=(V_{\eps,1},V_{\eps,2})$ is given
by the formula $V_\eps(t;x)=\nabla_x^\perp H_\eps(t;x)$, with
$\nabla_x^\perp:=[\partial_{x_2},-\partial_{x_1}]$  and
$H_\eps(t;x):=H(t,x,\eps x)$, where
\begin{align}\begin{aligned}\label{wzory2}
&H(t,x,y)=\sum_{i=1}^N[a_i(t;y)\cos(k_i\cdot
x)+b_i(t;y)\sin(k_i\cdot x)].
 \end{aligned}\end{align}
The processes $a_i(t,y)$ and $b_i(t,y)$ are described by \eqref{roz},
with the multi-index $\jj$ replaced by $i$. The formulas
\eqref{genOrnstein} and \eqref{miara} for  the generator and the
invariant measure are modified in an obvious fashion, with $S=N$.


We can write $V_\eps(t,x)=V(t,x,\eps x)$ with
\begin{equation}
\label{V}
V(t,x,y)=W(t,x,y)+\eps U(t,x,y),
\end{equation}
where
\begin{align}\begin{aligned}\label{wzory3}
&W(t,x,y)=\sum_{i=1}^Nk_i^\perp[-a_i(t;y)\sin(k_i\cdot
x)+b_i(t;y)\cos(k_i\cdot x)],\\
&U(t,x,y)=\sum_{i=1}^N[a_{i,y}^\perp(t;y)\cos(k_i\cdot
x)+b_{i,y}^\perp(t;y)\sin(k_i\cdot x)].
 \end{aligned}\end{align}
We shall denote 
\begin{equation}
\label{k-perp}
k_i^\perp:=(k_{i,2},-k_{i,1})
\end{equation}
 for
$k_i=(k_{i,1},k_{i,2})$ and  $a_{i,y}^\perp(t;y):=\nabla^\perp_ya_i(t;y)$, and likewise
for $b_{i,y}^\perp(t;y)$.

\subsection{Auxiliary dynamics}

For any $y\in\mbR^d$ we consider the auxiliary dynamics
$z(t,y)$ given by
\begin{equation}
\begin{aligned}\label{dynPom}
&\frac{dz(t,y)}{dt}= W\left(t,z(t,y),y\right),\\
&z(0,y)=0.
\end{aligned}
\end{equation}
Define $\left(\tilde\fa(t;y)\right)_{t\ge0}$ an $\R^{2N}$-valued
process, given by $\tilde\fa(t;y):=\left(\tilde a_i(t;y),\tilde
  b_i(t;y)\right)_{i=1,\ldots,N}$, with
\begin{align}
&\tilde a_i(t;y):=a_i(t;y)\cos(k_i\cdot
z(t,y))+b_i(t;y)\sin(k_i\cdot z(t,y)),\nonumber\\
&\tilde b_i(t;y):=-a_i(t;y)\sin(k_i\cdot
z(t,y))+b_i(t;y)\cos(k_i\cdot z(t,y)).\label{tyldab}
\end{align}
 Equation \eqref{dynPom} can be rewritten in the form
\begin{equation}
\begin{aligned}
&\frac{dz(t,y)}{dt}=\sum_{i=1}^N\tilde b_{i}\left(t,y\right)k_i^\perp,\\
&z(0,y)=0,
\end{aligned}
\end{equation}
For any $k=(k_1,k_2),\ell=(\ell_1,\ell_2)\in \R^2$ we let
\begin{align}\label{kij}
\delta(k,\ell)=k\cdot \ell^\perp=k_1\ell_2-k_2\ell_1.
\end{align}
A simple application of It\^o formula shows that  the components of $\tilde\fa(t;y)$
satisfy the following It\^o stochastic differential equation
\begin{equation}\label{dyf2}
\begin{aligned}
&d\tilde a_i(t)=\left\{-\alpha_i\tilde
a_i(t)+\sum_{j=1}^N\delta(k_{i},k_j)\tilde b_j(t)\tilde
b_i(t)\right\} dt+\sqrt{2\alpha_i}\sigma_i d\tilde w_{i,a}(t),\\
 &d\tilde
b_i(t)=\left\{-\alpha_i\tilde
b_i(t)-\sum_{j=1}^N\delta(k_{i},k_j)\tilde b_j(t)\tilde
a_i(t)\right\} dt+\sqrt{2\alpha_i}\sigma_i d\tilde
w_{i,b}(t),\quad i=1,\ldots,N.\\
\end{aligned}\end{equation}
Here $\tilde w_{i,a}$, $\tilde w_{i,b}$, $i=1,\ldots,N$ are
i.i.d. standard, one dimensional Brownian motions. To shorthen the
notation we have omitted writing the argument $y$. Let $ \fa\in\bbR^{2N}$. We denote by
$\tilde\fa^\fa(t)=\left(\tilde a_i^\fa(t),\tilde b_i^\fa(t)\right)$
the solution of \eqref{dyf2} satisfying $\tilde\fa^\fa(0)=\fa$.


 The generator of the diffusion \eqref{dyf2} is given by 
\begin{equation}
\mathcal{L}_yF=\mathbb L^yF+\mathfrak{w}\cdot D F
\label{gen}, \quad F\in C^2\left(\mbR^{2N}\right),
\end{equation}
where
\begin{equation}\label{obrot}
D F:=\sum_{i=1}^Nk_i\mathcal{R}_iF,\qquad
\end{equation}
with
the differential operator
\begin{equation}\label{Rj}
\mathcal{R}_i F:=\left( b_i\partial_{a_i}-
a_i\partial_{b_i}\right)F,\qquad F\in C^2(\mbR^{2N}).
\end{equation}
The mapping $\mathfrak{w}=(\mathfrak{w}_1, \mathfrak{w}_2):\bbR^{2N}\to\bbR^2$ is given by
\begin{equation}\label{polew}
\mathfrak{w}(\fa)=\sum_{i=1}^Nk_i^{\perp}b_i,\quad \fa=(a_i,b_i)_{i=1,\ldots,N}.
\end{equation}
Obviously the components of $\mathfrak{w}$ belong to 
$L^p(\nu_*^y)$ for any $p\in[1,+\infty)$.   Define also
\begin{equation}\label{obrot-p}
D^\perp F:=\sum_{i=1}^Nk_i^\perp\mathcal{R}_iF.
\end{equation}

Let $(P_t^{y})_{t\geq0}$ be the transition semigroup of the diffusion
given by  \eqref{dyf2}. 
For any $p\in[1,+\infty)$
and a natural number $m$ let
$W^{p,m}(\nu_{*}^y)$  be the Sobolev space made of those
$F\in L^p(\nu_*^y)$ whose partials of order $m$ are $L^p$ integrable
with respect to $\nu_*^y$ equipped with the norm
$$
\|F\|_{W^{p,m}(\nu_{*}^y)}:=\left\{\sum_{\ell=0}^m\|\nabla_{\frak a}^\ell F\|_{L^p(\nu_{*}^y)}^p\right\}^{1/p}.
$$
Here $\nabla_{\frak a}^\ell $ denotes the $\ell$-th order, derivative
tensor with respect to the variable ${\frak a}
$.
By $W^{p,m}(\bbR^{2N})$,   $W^{p,m}_{loc}(\bbR^{2N})$ we denote
correspondingly the standard Sobolev space with
respect to the ''flat'' Lebesgue measure and the space of functions
that belong to the respective Sobolev space on any ball.

  Define the set
\begin{equation}\label{C0}
\mathcal{C}_0:=\bigcap_{p\geq 1}W^{p,2}(\nu_{*}^y).
\end{equation}
A simple calculation shows that
\begin{equation}
\label{012706-18}
\int_{\R^{2N}}\mathfrak{w}\cdot D F G d\nu_*^y =-\int_{\R^{2N}}F\mathfrak{w}\cdot D G d\nu_*^y
\end{equation}
for any $F,G\in \mathcal{C}_0$. As a
result from \eqref{inv} and \eqref{012706-18} (with $G\equiv1$) we
conclude that
\begin{equation}
\label{inv1}
\int_{\R^{2N}}\mathcal{L}_yF d\nu_*^y=0
\end{equation}
and, thanks to \eqref{spectralgapOU2},
\begin{equation}
\label{inv2}-\langle {\cal
L}_yF(\cdot;y),F(\cdot;y)\rangle_{\nu_*^y}\geq\gamma_0\|F(\cdot;y)\|^2_{L^2(\nu_{*}^y)},\quad
y\in\mbR^2,\mbox{ provided that }\int_{\R^{2N}}F d\nu_*^y=0
\end{equation}
for any $F\in \mathcal{C}_0$. In particular, from \eqref{inv1} we
obtain that $\nu_*^y$ is an invariant measure for the dynamics
described by \eqref{dyf2}. In consequence, 
\begin{equation*}
\int_{\R^{2N}}P_t^{y} F d\nu_*^y=\int_{\R^{2N}} F d\nu_*^y,\quad t\ge 0,\,y\in\bbR^2,
\end{equation*}
for any bounded and measurable $F$.
The semigroup extends therefore to a Markovian contraction $C_0$-semigroup on
$L^p(\nu_*^y)$ for any $p\in[1,+\infty)$. 
We have the following.
\begin{Propozycja}[see \cite{Ksiazka} Section 12.3]\label{Lem542}
${\cal C}_0$ is a core of the generator ${\cal L}_y$ of the semigroup
$(P_t^{y})_{t\geq0}$. On this set, the generator coincides with the differential operator
 given by \eqref{gen}.
\end{Propozycja}

As a direct conclusion from Proposition \ref{Lem542} and estimate \eqref{inv2} we obtain also.
\begin{Propozycja}\label{spectralneStwr} Fix $y\in\mbR^2$.
Assume that a function $F:\mbR^{2N}\to\mbR$, is such that $F\in
L^2(\nu_{*}^y)$ and
\begin{equation}\label{4.21}
\int_{\mbR^{2N}} Fd\nu_{*}^y=0.
\end{equation}
Then,
\begin{align}\label{spectralgap2}\begin{aligned}
&||P_t^{y}F||_{L^2(\nu_{*}^y)}\leq e^{-\gamma_0
t}||F||_{L^2(\nu_{*}^y)},\quad t\ge0.\\
\end{aligned}\end{align}
where $\gamma_0$ was introduced in $\eqref{gamma}$.
\end{Propozycja}
\begin{remark} 
Fix $y\in\mbR^2$.  Suppose that  $F$ satisfies   \eqref{4.21}.
Applying an interpolation argument  we conclude, from
\eqref{spectralgap2} and the fact that $P_t^{y}$ is contraction in
both $L^1(\nu_*^y)$ and $L^\infty(\nu_*^y)$,
that for any $p\in(1,+\infty)$ there exists  $\gamma(p)\in(0,\gamma_0)$
such that
\begin{equation}\label{pspectral}
||P_t^{y}F(\cdot;y)||_{L^p(\nu_{*}^y)}\leq
e^{-\gamma(p)t}||F(\cdot;y)||_{L^p(\nu_{*}^y)},\quad \ t\geq0.
\end{equation}
\end{remark}


\subsection{Corrector} \label{korekotsek}

Since 
$$ 
\int_{\mbR^{2N}}\mathfrak{w}_q d\nu_{*}^y=0 , \quad q=1,2
$$
 for any
$y\in\bbR^2$,
we can define 
 the corrector in direction $e_q=(\delta_{1,q},\delta_{2,q})$ by letting
$$
\chi_q(\cdot;y):=\int_0^{+\infty}P_t^y \mathfrak{w}_qdt,\quad q=1,2,\,y\in\bbR^2.
$$
Thanks to \eqref{pspectral} it  belongs to the $L^p(\nu_{*}^y)$-domain of the
generator $\mathcal{L}_y$
and it is the unique solution of the problem
\begin{equation}\label{011702}
 -{\cal L}_y\chi_q(\cdot;y)=\mathfrak{w}_q\quad\mbox{and}\quad  \int_{\mbR^{2N}}\chi_q(\cdot;y)d\nu_{*}^y=0.
\end{equation}

From the standard regularity results for diffusions with smooth coefficients it
follows that corrector $\chi_q(\cdot;y)$
belongs to the Sobolev space  $W^{m,p}_{loc}(\bbR^{2N})$ for any
$p\in[1,+\infty)$ and $m\ge1$. 
Far less trivial is the regularity of the corrector in the $y$-variable.
From Theorem \ref{lab} below we can conclude the following.
\begin{Propozycja}
\label{prop012706-18}
We have
$\chi_q\in W^{2,p}_{loc}(\bbR^{2N+2})$ for any $p\in[1,+\infty)$ and $q=1,2$.
\end{Propozycja}

\section{Proof of Theorem \ref{glowne}}

\label{sec5}

To abbreviate the notation for a given function $G:\bbR^{2N+2}\to\bbR$
we write 
\begin{equation}\label{ozn}
G^{(\ep)}\left(t\right):= \tilde
G\left(\tep,\bar x_\ep(t),x_\ep(t)\right),
\end{equation}
where $\tilde
G$ is given by \eqref{tild}, $x_\ep(t)$ is the scaled trajectory as in
\eqref{poleweup}, with $s=0$, $x_0=0$, and $\bar x_\ep(t):=\eps^{-1}x_\ep(t)$.
Using the It\^{o}-Krylov formula (see \cite{Krylov}, Theorem 1,
p. 122) and the above convention we obtain
\begin{align}\label{chitylda}\begin{aligned}
& d\left[\eps
  \chi^{(\ep)}_{q}\left(t\right)\right]=\Bigg\{\frac{1}{\eps}({\cal
      L}\chi_{q})^{(\ep)}\left( t\right)+  U_{\ep}\left(t\right)\cdot(D\chi_q)^{(\ep)}(t)+\bbV^{(\ep)}(t)\cdot\chi^{(\ep)}_{q,y}\left(
t\right)\\
&+\sum_{i=1}^NV^{(\ep)}\left(t\right)\cdot\bigg[\chi_{q,a_i}^{(\ep)}\left(
t\right) a_{i,y}^{(\eps)}(t)
+\chi_{q,b_i}^{(\ep)}\left(
t\right) b_{i,y}^{(\eps)}(t)
\bigg]\Bigg\}dt+{\cal
M}_q^{(\ep)}(dt),\quad q=1,2,
\end{aligned}\end{align} 
Here the processes $\chi^{(\ep)}_{q,a_i}(t)$,  $\chi^{(\ep)}_{q,b_i}(t)$, $\chi^{(\ep)}_{q,y}(t)$ are formed from
the partials $\chi_{q,a_i}({\frak a},y)$, $\chi_{q,b_i}({\frak a},y)$ and $\chi_{q,y}({\frak a},y)$ using the convention introduced in  \eqref{tild} and \eqref{ozn}. 
The processes $({\cal
      L}\chi_{q})^{(\ep)}\left( t\right)$, $(D\chi_q)^{(\ep)}(t)$ are obtained from the fields  $\widetilde{{\cal
      L}\chi_{q}}\left( t,x,y\right)$ and $D\tilde\chi_{q}(t,x,y)$ (operator $D$ is defined in \eqref{obrot}).
We let
\begin{equation}
\label{a-eps}
{\fa}^{(\ep)}(t)=\left(  a_{i}^{(\ep)}(t),
b_{i}^{(\ep)}(t)\right)_{i=1,\dots,N},
\end{equation}
where 
\begin{align}
\label{a-b}
&a_{i}^{(\ep)}(t):= a_i(t, x_{\eps}(t))\cos(k_i\cdot \bar x_{\eps}(t))+b_i(t, x_{\eps}(t))\sin(k_i\cdot \bar x_{\eps}(t)),\\
&
 b_{i}^{(\ep)}(t):=-a_i(t, x_{\eps}(t))\sin(k_i\cdot \bar x_{\eps}(t))+b_i(t, x_{\eps}(t))\cos(k_i\cdot \bar x_{\eps}(t) )).
\nonumber
\end{align}
Similarly we define $
{\fa}^{(\ep)}_y(t)=\left(  a_{i,y}^{(\ep)}(t),
b_{i,y}^{(\ep)}(t)\right)_{i=1,\dots,N},
$ using the processes $a_{i,y}(t,y)$ and $b_{i,y}(t,y)$ - the
derivatives w.r.t. variable $y$ of $a_{i}(t,y)$ and $b_{i}(t,y)$ correspondingly.

The martingale term is given by
\begin{equation}
\label{mart-1} {\cal
M}_q^{(\ep)}(dt):=\sum_{i=1}^N\left(\sigma_i\sqrt{2\al_i }\right)(x_{\eps}(t))\Big[\partchia\left(
t\right)dw_{i,a}(t)+ \partchib\left( t\right)dw_{i,b}(t)\Big]
\end{equation}
 with  $w_{i,a}(t), w_{i,b}(t)$, $i=1,\ldots,N$  independent standard Brownian motions. 
 
 Using the fact that $-\widetilde{{\cal L}\chi_q}(t,x,y)=W_q(t,x,y)$ we obtain
 \begin{align}\label{ixy}\begin{aligned}
& x_{\ep,q}(t)=\eps
  \chi^{(\ep)}_{q}\left(0\right)-\eps
  \chi^{(\ep)}_{q}\left(t\right)+\sum_{\iota=1}^2\int_0^t{\cal
    W}^{(\ep,\iota)}_{q} \left( s\right)ds +\int_0^t{\cal M}_q^{(\ep)}\left(ds\right),\quad q=1,2,
\end{aligned}\end{align}
 where ${\cal W}^{(\ep,\iota)}_q$ are formed  (by means of 
 \eqref{tild}) from
    \begin{align}
    \label{cWj}
 &{\cal W}_q^1(t,x,y):=U_q(t,x,y)+
U(t,x,y)\cdot D\tilde \chi_{q}(t,x,y) \nonumber \\
&
+\sum_{i=1}^NW(t,x,y)\cdot\left[\tpartchiaB(t,x,y)\tilde a_{i,y}(t,x,y)
+\tpartchibB(t,x,y)\tilde b_{i,y}(t,x,y)
\right],\nonumber\\
&{\cal W}_q^{2}(t,x,y):=V(t,x,y)\cdot\tilde\chi_{q,y}(t,x,y)\\
&+\ep\sum_{i=1}^NU(t,x,y)\cdot\left[\tpartchiaB(t,x,y)\tilde a_{i,y}(t,x,y)
+\tpartchibB(t,x,y)\tilde b_{i,y}(t,x,y)
\right].\nonumber
\end{align}
Observe that (cf \eqref{wzory3}) that the terms constituting ${\cal W}_q^{1}$  contain the fields $\tilde a_{i,y}(t,x,y)$ and $\tilde b_{i,y}(t,x,y)$
and are of apparent order of magnitude $O(1)$. They are not of the
form \eqref{ozn} and therefore require a separate treatment.
 We are going to deal with these terms in Section
 \ref{szac}. Expressions included in  ${\cal W}_q^{2}$
are either the  terms of order $O(1)$ that are of the form
\eqref{ozn}, or terms that are of apparent order of magnitude
$O(\eps)$. We shall deal with them
in Section \ref{twrergosek}.


\subsection{Term   ${\cal W}_q^{1}$}

\label{szac}

To avoid using multitude of constants appearing in our estimates, for
any two expressions $f,g:A\to[0,+\infty)$, where $A$ is some set, we
shall write $f\preceq g$ iff there exists $C>0$ such that
$f(a)\le C g(a)$, $a\in A$.  We shall also write $f\approx g$ iff
$f\preceq g$ and $g\preceq f$.

Comparing \eqref{wzory3} with the first  formula of \eqref{cWj}  we conclude that 
\begin{equation}\label{dgwai-1}
{\cal
  W}_q^{(1,\eps)}(t)=\sum_{i,j}a_{i,y_j}^{(\eps)}(t)F_{i,j}^{(\eps)}(t)+\sum_{i,j}b_{i,y_j}^{(\eps)}(t)G_{i,j}^{(\eps)}(t)
\end{equation}
where the summations extend over $i=1,\ldots,N$, $j=1,2$,
$a_{i,y_j}^{(\eps)}(t)$, $b_{i,y_j}^{(\eps)}(t)$ are given by
\begin{align}
\label{a-b-y}
&a_{i,y_j}^{(\ep)}(t):= a_{i,y_j}(t, x_{\eps}(t))\cos(k_i\cdot \bar x_{\eps}(t))+b_{i,y_j}(t, x_{\eps}(t))\sin(k_i\cdot \bar x_{\eps}(t)),\\
&
 b_{i,y_j}^{(\ep)}(t):=-a_{i,y_j}(t, x_{\eps}(t))\sin(k_i\cdot \bar x_{\eps}(t))+b_{i,y_j}(t, x_{\eps}(t))\cos(k_i\cdot \bar x_{\eps}(t) ))
\nonumber
\end{align}
and 
$
F_{i,j}^{(\eps)}(t)$, $G_{i,j}^{(\eps)}(t)$
 are the processes obtained  (using \eqref{ozn})
from 
\begin{align*}
&F_{i,j}({\frak a},y)=e_{q,j}^\perp+D_j^\perp \chi_q({\frak
  a},y)+{\frak w}_j({\frak a})\chi_{q,a_i}({\frak a},y),\\
&
G_{i,j}({\frak a},y)={\frak w}_j({\frak a})\chi_{q,b_i}({\frak a},y),\quad
i=1,\ldots,N,\,j=1,2,
\end{align*}
 with $e_q^\perp$, $\frak w$ and $D^\perp$ given by \eqref{k-perp},  \eqref{polew} and
 \eqref{obrot-p} respectively.
Using Theorem \ref{lab} and Proposition \ref{Gaussowskie} one can
conclude that these functions satisfy
\begin{equation}\label{zalF}
\sup_{y\in\mbR^2}\left(\| F_{i,j}(\cdot;y)\|_{L^p(\nu_{*}^y)}+\| G_{i,j}(\cdot;y)\|_{L^p(\nu_{*}^y)}\right)<+\infty,\quad
j=1,2,i=1,\dots,N
\end{equation}
for any $p\in[1,+\infty)$.

To average expressions of the form \eqref{dgwai-1} we represent them, using appropriately defined correctors, as
functionals of the process ${\frak a}^{(\eps)}(t)$ (cf \eqref{a-eps}). This enables us to
apply our averaging result, see Theorem \ref{ergodyczne} below, to identify
their appropriate limits, as $\eps\to0$.


We introduce the correctors
$\Theta_{i,j}^{(1)},\Theta^{(2)}_{i,j}:\mbR^{2N+2}\to\bbR,$
which are the solutions of the following system of equations
\begin{align}\label{Theduz}\begin{aligned}
& \left[{\cal L}_y- \al_i(y) \right]\Theta^{(1)}_{i,j}(\cdot;y)-(k_i \cdot
\mathfrak{w}) \Theta^{(2)}_{i,j}(\cdot;y)=F_{i,j}(\cdot;y),\\
& \left[{\cal L}_y- \al_i(y) \right]\Theta^{(2)}_{i,j}(\cdot;y)+(k_i \cdot
\mathfrak{w}) \Theta^{(1)}_{i,j} (\cdot;y)=G_{i,j}(\cdot;y),
\end{aligned}\end{align}
such that
\begin{equation}\label{przestrzen2}
\left\|\Theta^{(\iota)}_{i,j}(\cdot;y)\right\|_{L^p(\nu_{*}^y)}<+\infty,\quad
\iota,j=1,2,\ i=1,\dots,N, p\in[1,+\infty),\,y\in \bbR^2.
\end{equation}
To solve the above system, note that
the function $\Psi_{i,j}:\mbR^{2N+2}\to\mathbb C$
$$
\Psi_{i,j}:=\Theta^{(1)}_{i,j}+\textrm{i}\Theta^{(2)}_{i,j}
$$
satisfies
\begin{equation}\label{wyjsciowe}
\left[{\cal L}_y- \al_i(y) \right]\Psi_{i,j}(\cdot;y)+\textrm{i}(k_i
\cdot \mathfrak{w}) \Psi_{i,j}(\cdot;y)=K_{i,j}(\cdot;y),
\end{equation}
with $K_{i,j}:= F_{i,j} +\textrm{i} G_{i,j}$. Here
$\textrm{i}=\sqrt{-1}$ denotes the imaginary unit and ${\frak w}$ is
given by \eqref{polew}.

Using the Feynman-Kac formula we obtain that the solution \eqref{wyjsciowe} is given by
\begin{align}\label{Feynamnkac}\begin{aligned}
&\Psi_{i,j}(\fa;y)=\int_0^{+\infty}\bbE\Bigg[\exp\left\{-\al_i(y)t+\textrm{i}\int_0^t(k_i
\cdot \mathfrak{w}(\tilde\fa^{\fa}(s;y))ds\right\}K_{i,j}(\tilde\fa^{\fa}(t;y);y)\Bigg]dt,
\end{aligned}\end{align}
$i=1,\dots,N$, $j=1,2$. Here $\tilde\fa^\fa(t)=\left(\tilde a_i^\fa(t),\tilde b_i^\fa(t)\right)$
the solution of \eqref{dyf2} satisfying $\tilde\fa^\fa(0)=\fa$. The
improper integral above converges absolutely, thanks to assumption \eqref{gamma}.


The following lemma allows us to replace expressions containing
processes $a_{i,y_j}^{(\eps)}(t)$ and  $b_{i,y_j}^{(\eps)}(t)$ by
functionals of the processes $a_{i}^{(\eps)}(t)$ and  $b_{i}^{(\eps)}(t)$.
\begin{lem}\label{pochodnelem}
Assume that $F_{i,j}$, $G_{i,j}$, $i=1,\ldots,N$, $j=1,2$ satisfy
\eqref{zalF} and $\Theta^{(\iota)}_{i,j}$ are the respective solutions
of \eqref{Theduz}. Suppose also that the processes $\left(a_{i,y_j}^{(\ep)}(t)\right)_{t\ge0}$, $\left(b_{i,y_j}^{(\ep)}(t)\right)_{t\ge0}$  are given by \eqref{a-b-y} and $\left(F_{i,j}^{(\eps)}(t)\right)_{t\ge0}$, $\left( G_{i,j}^{(\eps)}(t)\right)_{t\ge0}$  are obtained from $F_{i,j}$, $G_{i,j}$ using \eqref{ozn}. Then, 
\begin{align}\label{lematPoch}\begin{aligned}
&\int_0^t\left[ a_{i,y_j}^{(\ep)}(s)F_{i,j}^{(\eps)}(s)+
b_{i,y_j}^{(\ep)}(s) G_{i,j}^{(\eps)}(s)\right]ds\\
&=\int_0^t \left\{\partial_{y_j} \al_i(x_\ep(s))
\left(\Theta^{(1,\ep)}_{i,j}(s)a_{i}^{(\ep)}(s)+\Theta^{(2,\ep)}_{i,j}(s)b_{i}^{(\ep)}(s)\right)\right\}ds+\eps
\tilde N^{(\ep)}(t),
\end{aligned}\end{align} where $\ep\tilde N^{(\ep)}(t)$ is a negligible semi-martingale
i.e. for any $T>0$ we have
\begin{equation}\label{zanied}
\lim_{\ep\to0}\ep\bbE\left[\sup_{0\leq t\leq T} \left|\tilde
N^{(\ep)}(t)\right|\right]=0.
\end{equation}
\end{lem}
\begin{proof}
From \eqref{roz} we obtain
\begin{align}\label{calkia}\begin{aligned}
&a_{i,y_j}(t;y)=\sqrt{2\alpha_i(y)}\gamma_{i,j}(y)\int_{-\infty}^te^{-\alpha_i(y)(t-s)}dw_{i,a}(s)\\\
&-\sqrt{2\alpha_i(y)}\sigma_i(y)\int_{-\infty}^t(t-s)\alpha_{i,y_j}(y)e^{-\alpha_i(y)(t-s)}dw_{i,a}(s).
\end{aligned}\end{align} 
with
$$
\gamma_{i,j}(y):=\sigma_{i,y_j}(y)+\frac{\alpha_{i,y_j}(y)}{2\alpha_i(y)}\sigma_i(y),\quad
y\in\bbR^2.
$$
Its It\^o stochastic differential can be written in the form
\begin{equation}\label{OrnsteinGrad}
\begin{aligned}
&da_{i,y_j}(t;y)=-\left[\alpha_{i,y_j}(y)a_{i}(t;y)+\alpha_i(y)a_{i,y_j}(t;y)\right]dt+\sqrt{2\alpha_i(y)}\gamma_{i,j}(y)dw_{i,a}(t).
\end{aligned}
\end{equation}
 Similar formulas hold for $\left(b_{i,y_j}(t;y)\right)$. Recall that (cf \eqref{tild} and \eqref{wzory3})
$$
\Theta^{(\iota,\ep)}_{i,j}(s):=\tilde\Theta^{(\iota)}_{i,j}\left(\frac{t}{\eps^2},\bar
x_\eps(t),x_\eps(t)\right),\quad W^{(\eps)}(t):=W \left(\frac{t}{\eps^2},\bar
x_\eps(t),x_\eps(t)\right)
$$
and likewise for $U^{(\eps)}(t)$.

From the Leibnitz rule applied to the respective stochastic
differentials and straightforward (but rather lengthy calculation) we obtain
\begin{align}\begin{aligned}\label{costam}
&\int_0^t \left\{ \al_{i,y_j}(x_\ep(s))
\left(\Theta^{(1,\ep)}_{i,j}(s) a_{i}^{(\ep)}(s)+\Theta^{(2,\ep)}_{i}(s) b_{i}^{(\ep)}(s)\right)\right\}ds+\eps  \tilde N^{(\eps)}(t)\\
&=\int_0^t\left[\left[{\cal L}- \al_i(x_\ep(s))
\right]\Theta^{(1,\ep)}_{i,j}(s)-(k_i \cdot
\bbW^{(\ep)}(s))\Theta^{(2,\ep)}_{i,j}(s)\right] a_{i,y_j}^{(\ep)}(s)ds\\
&+\int_0^t\left[\left[{\cal L}- \al_i(x_\ep(s))
\right]\Theta^{(2,\ep)}_{i,j}(s)+(k_i \cdot W^{(\ep)}(s))
\Theta^{(1,\ep)}_{i,j}(s)\right] b_{i,y_j}^{(\ep)}(s)ds.
\end{aligned}\end{align} 
We have denoted by
\begin{align}\begin{aligned}\label{tldN}
&\tilde N^{(\eps)}(t):=\ep\left[\Theta^{(1,\ep)}_{i,j}(s)
  a_{i,y_j}^{(\ep)}(s)+\Theta^{(2,\ep)}_{i,j}(s)
  b_{i,y_j}^{(\ep)}(s)\right]_{s=0}^t- \int_0^t\Gamma^{(\eps)}(s)ds-M^{(\eps)}(t),
\end{aligned}\end{align} 
where $
[f(s)]_{s=0}^t:=f(t)-f(0)$. We let $ M_0^{(\eps)}=0$ and
\begin{align}\begin{aligned}\label{N_t1}
&
\Gamma^{(\eps)}(t):=
(k_i\cdot U^{(\ep)}(t))\big[ \Theta^{(1,\ep)}_{i,j}\left( t\right)
b_{i,y_j}^{(\ep)}\left(t\right)-\Theta^{(2,\ep)}_{i,j}\left( t\right)
a_{i,y_j}^{(\ep)}\left(t\right)\big]
\\
&
+a_{i,y_j}^{(\ep)}\left(t\right)\Bigg\{U^{(\ep)}\left(t\right)\cdot(D\Theta_{i,j}^{(1)})^{(\ep)}(t)+V^{(\ep)}\left(t\right)\cdot\Big\{\sum_{i'}\bigg[\Theta^{(1,\ep)}_{i,j,a_{i'}}(t)
a_{i',y}^{(\ep)}\left(t\right)
+\Theta^{(1,\ep)}_{i,j,b_{i'}}(t)
b_{i',y}^{(\ep)}\left(t\right)\bigg]+\Theta^{(1,\ep)}_{i,j,y}(t)\Big\}\Bigg\}\\
&
+b_{i,y_j}^{(\ep)}\left(t\right)\Bigg\{U^{(\ep)}\left(t\right)\cdot(D\Theta_{i,j}^{(2)})^{(\ep)}(t)+V^{(\ep)}\left(t\right)\cdot\Big\{\sum_{i'}\bigg[\Theta^{(2,\ep)}_{i,j,a_{i'}}(t)
a_{i',y}^{(\ep)}\left(t\right)
+\Theta^{(2,\ep)}_{i,j,b_{i'}}(t)
b_{i',y}^{(\ep)}\left(t\right)\bigg]+\Theta^{(2,\ep)}_{i,j,y}(t)\Big\}\Bigg\}\\
&
+2\ep\left(\al_i\gamma_{i,j}\sigma_{i}\right)(x_\ep(t))\Big\{\cos(k_i \cdot
\bar x_\ep(t))\Theta^{(1,\ep)}_{i,j,a_{i}}(t)-\sin(k_i \cdot
\bar x_\ep(t))\Theta^{(2,\ep)}_{i,j,a_{i}}(t)\\
&
+\sin(k_i \cdot \bar
x_\ep(t))\Theta^{(1,\ep)}_{i,j,b_{i}}(t)+\cos(k_i \cdot \bar
x_\ep(t))\Theta^{(2,\ep)}_{i,j,b_{i}}(t)\Big\}
+V^{(\eps)}(t)\cdot\Big[\Theta^{(1,\ep)}_{i,j}(t)
a_{i,y,y_j}^{(\eps)}(t)+\Theta^{(2,\ep)}_{i,j}(t)
b_{i,y,y_j}^{(\eps)}(t)\Big],
\end{aligned}
\end{align}
and
\begin{align}\begin{aligned}\label{N_t12}
&
dM^{(\eps)}(t):=\sum_{i'}
a_{i,y_j}^{(\ep)}\left(t\right)(\sigma_{i'}\sqrt{2\al_{i'}})(x_\ep(t))\left[\Theta^{(1,\ep)}_{i,j,a_{i'}}(t)dw_{i',a}(t)+
\Theta^{(1,\ep)}_{i,j,b_{i'}}(t)dw_{i',b}(t)\right]\\
&+\sum_{i'}
b_{i,y_j}^{(\ep)}\left(t\right)(\sigma_{i'}\sqrt{2\al_{i'}})(x_\ep(t))\left[\Theta^{(2,\ep)}_{i,j,a_{i'}}(t)dw_{i',a}(t)+
\Theta^{(2,\ep)}_{i,j,b_{i'}}(t)dw_{i',b}(t)\right]\\
&+\Theta^{(1,\ep)}_{i,j}(t)\bigg\{(\gamma_{i,j}\sqrt{2\al_{i}})(x_\ep(t))\big[\cos(k_i \cdot
 \bar x_\ep(t))dw_{i,a}(t)+\sin(k_i \cdot \bar x_\ep(t)) dw_{i,b}(t)\big]\bigg\}\\
&+\Theta^{(2,\ep)}_{i,j}(t)\bigg\{(\gamma_{i,j}\sqrt{2\al_{i}})(x_\ep(t))\big[-\sin(k_i \cdot
 \bar x_\ep(t))dw_{i,a}(t)+\cos(k_i \cdot  \bar x_\ep(t)) dw_{i,b}(t)\big]\bigg\}.
\end{aligned}
\end{align}
The processes $a_{i,y_j,y}^{(\ep)}\left(t\right)$ and $b_{i,y_j,y}^{(\ep)}\left(t\right)$ appearing in the above formulas are given by analogues 
of \eqref{a-b-y}, where the first derivatives in $y$ are replaced by  the respective second derivatives.

Substituting from  \eqref{wyjsciowe} we obtain equality \eqref{lematPoch}.
In order to prove  \eqref{zanied} we consider the three terms that appear on the right hand side of \eqref{tldN}.
Concerning the boundary terms, by the Cauchy-Schwarz inequality, we can
estimate as follows
\begin{equation}
\label{010704}
\ep\bbE\left[\sup_{0\leq t\leq
T}|\Theta^{(1,\ep)}_{i,j}(t)a_{i,y_j}^{(\ep)}(t)|\right]\leq
\ep\left[\bbE\sup_{0\leq t\leq
T}\left(a_{i,y_j}^{(\ep)}(t)\right)^2\right]^{1/2}\left[\bbE\sup_{0\leq
t\leq T}\left(\Theta^{(1,\ep)}_{i,j}(t)\right)^2\right]^{1/2}.
\end{equation}
Furthermore
\begin{equation}\label{gaus}
\sup_{0\leq t\leq T}\left|a_{i,y_j}^{(\ep)}(t)\right|\preceq\sup_{0\leq
t\leq T,\,y\in
\mbR^2}\left\{\left|a_{i,y_j}(t/\ep^2;y)|+|b_{i,y_j}(t/\ep^2;y)\right|\right\}.
\end{equation}
The fields $\left(a_{i,y_j}(t;y)\right)_{(t,y)\in\bbR^3}$, $\left(b_{i,y_j}(t;y)\right)_{(t,y)\in\bbR^3}$ are Gaussian.  Using the results of Section
\ref{sec12.1} we conclude  that for any
  $\gamma\in(0,1)$ we have
\begin{equation}\label{pud}
\bbE\left[\sup_{0\leq t\leq
T}\left(a_{i,y_j}^{(\ep)}(t)\right)^2\right]\preceq\ep^{-\gamma},\quad \eps\in(0,1).
\end{equation}
From \eqref{pud} we obtain that
\begin{equation}\label{tigi}
\ep\bbE\left[\sup_{0\leq t\leq
T}|\Theta^{(1,\ep)}_{i,j}(t)a_{i,y_j}^{(\ep)}(t)|\right]\preceq
\ep^{1-\gamma/2}\left\{\bbE\left[\sup_{0\leq t\leq
T}\left(\Theta^{(1,\ep)}_{i,j}(t)\right)^2\right]\right\}^{1/2}.
\end{equation}
Applying Lemma \ref{SUPREMUM} we conclude that the right hand side vanishes, as $\eps\to0$. Therefore
\begin{equation}\label{pommis}
\lim_{\ep\to0}\ep\bbE\left[\sup_{0\leq t\leq
T}|\Theta^{(1,\ep)}_{i,j}(t)a_{i,y_j}^{(\ep)}(t)|\right]=0,\quad T>0.
\end{equation}


Concerning the bounded variation term on the right hand side of \eqref{tldN} note that the  last expression on the right hand side of
\eqref{N_t1} can be estimated as follows
\begin{align*}
&\bbE\left[\sup_{0\leq t\leq T}\left| \int_0^t
V^{(\eps)}(s)\cdot\Big[\Theta^{(1,\ep)}_{i,j}(s)
a_{i,y,y_j}^{(\eps)}(s)+\Theta^{(2,\ep)}_{i,j}(s)
b_{i,y,y_j}^{(\eps)}(s)\Big]ds\right|\right]\\
&\preceq T\bbE\left\{\sup_{0\leq t\leq
T}\left\{|V^{(\eps)}(t)|\Big[|\Theta^{(1,\ep)}_{i,j}(t)
a_{i,y,y_j}^{(\eps)}(t)|+|\Theta^{(2,\ep)}_{i,j}(t)
b_{i,y,y_j}^{(\eps)}(t)|\Big]\right\}\right\}.
\end{align*}
From this point on the estimate can be conducted in a similar way as for the  boundary  terms and we conclude that
\begin{align*}
\lim_{\eps\to0}\eps\bbE\left[\sup_{0\leq t\leq T}\left| \int_0^t
V^{(\eps)}(s)\cdot\Big[\Theta^{(1,\ep)}_{i,j}(s)
a_{i,y,y_j}^{(\eps)}(s)+\Theta^{(2,\ep)}_{i,j}(s)
b_{i,y,y_j}^{(\eps)}(s)\Big]ds\right|\right]=0.
\end{align*}
The remaining terms appearing in the bounded variation expression on   the right hand side of
\eqref{N_t1}  can be dealt with similarly and we obtain
\begin{align*}
 \lim_{\eps\to0}\eps\bbE\left[\sup_{0\leq t\leq T}
\left| \int_0^t\Gamma^{(\eps)}(s)ds\right|\right]=0.
 \end{align*}
 For the martingale on the right hand side of \eqref{tldN}, its first
 term, given in the right hand side of \eqref{N_t12},
can be estimated first by Jensen's and then by Doob's inequality leading to
\begin{align}\label{tigii}\begin{aligned}
& \ep\bbE\left[\sup_{0\leq t\leq T} \left|\int_0^t\sum_{i'}
a_{i,y_j}^{(\ep)}\left(s\right)(\sigma_{i'}\sqrt{2\al_{i'}})(x_\ep(t))\left[\Theta^{(1,\ep)}_{i,j,a_{i'}}(s)dw_{i',a}(s)+
\Theta^{(1,\ep)}_{i,j,b_{i'}}(s)dw_{i',b}(s)\right]\right|\right]\\
&
\preceq \ep\left\{\bbE\left[ \int_0^T\sum_{i'}
\left(a_{i,y_j}^{(\ep)}\left(s\right)(\sigma_{i'}\sqrt{2\al_{i'}})(x_\ep(t))\right)^2\left[\left(\Theta^{(1,\ep)}_{i,j,a_{i'}}(s)\right)^2+
\left(\Theta^{(1,\ep)}_{i,j,b_{i'}}(s)\right)^2\right]ds\right]\right\}^{1/2}.
\end{aligned}\end{align} 
This expression vanishes, as $\ep\to0$,  thanks to the results of Section \ref{sec12.1}.
The other terms forming the martingale $M^{(\eps)}(t)$ can be estimated analogously. This ends the proof of \eqref{zanied}.
\end{proof}

\bigskip

Coming back to \eqref{ixy}, using \eqref{dgwai-1} together with \eqref{lematPoch}, we can write
\begin{align}\label{ixy2}\begin{aligned}
& x_{\ep,q}(t)=\int_0^t\widetilde{\cal
    W}^{(\ep)}_{q} \left(s\right)ds +\int_0^t{\cal M}_q^{(\ep)}\left(ds\right)+\eps
{\cal N}^{(\ep)}_q(t),\quad q=1,2,\end{aligned}\end{align} where $\widetilde{\cal
    W}^{(\ep)}_{q} \left(t\right)$ corresponds (via \eqref{ozn}) to the field
\begin{align*}\begin{aligned}
& \widetilde{\cal
    W}_q(t,x,y) :=\sum_{i,j}\left\{\al_{i,y_j}(y)
\left(\tilde\Theta^{(1)}_{i,j}(t,x,y)a_i(t,x,y)+\tilde\Theta^{(2)}_{i,j}(t,x,y)b_i(t,x,y)\right)\right\}\\
&+W(t,x,y)\cdot \chi_{q,y}(t,x,y).
\end{aligned}\end{align*}
The martingale part is given by \eqref{mart-1}. 
The semi-martingale $\eps{\cal N}^{(\ep)}_q(t)$, which turns out to be negligible (see \eqref{semi} below), is given by
\begin{eqnarray}\label{net}
&&{\cal N}^{(\ep)}_q(t):=
  \chi^{(\ep)}_q\left(0\right)-
  \chi^{(\ep)}_{q}\left( t\right)+\tilde N^{(\ep)}(t)+\int_0^t
\bbU_{\ep}(s)\cdot\tilde\chi_{q,y}^{(\ep)}(s)ds\nonumber
\\
&&
+\sum_{i,j}\int_0^tU_{\ep}(s)\cdot\left[\chi^{(\ep)}_{q,a_i}(s)
a_{i,y}^{(\ep)}(s)+ \chi^{(\ep)}_{q,b_i}(s)
b_{i,y}^{(\ep)}(s)\right]ds,
\end{eqnarray} where $\tilde N^{(\ep)}(t)$ is the process given by \eqref{tldN}.
From \eqref{zanied}, the results of Section \ref{sec12.1} and the argument used in the proof of Lemma \ref{pochodnelem} we conclude that
\begin{equation}\label{semi}
\lim_{\ep\to0}\ep\bbE\left[\sup_{0\leq t\leq
T}\left|{\cal N}^{(\ep)}(t)\right|\right]=0,
\end{equation}
with ${\cal N}^{(\ep)}(t):=({\cal N}^{(\ep)}_1(t), {\cal N}^{(\ep)}_2(t))$.



\subsection{Averaging lemma}

\label{twrergosek}
In this section we present a result, which allows us to average out the "fast variables", i.e. $t/\ep^2,x_\ep(t)/\ep$,
for
 processes of the form  $F^{(\eps)}(t):=\tilde
 F(t/\ep^2,x_\ep(t)/\ep,x_\ep(t))$, that  appear  in the significant terms of the decomposition \eqref{ixy2}. 
The following result allows to replace such terms by their averaged
 counterparts and therefore it  shall be crucial in the limit
identification argument for the trajectory process given by 
\eqref{ixy2}.
%
%
%
%
%
%
%
\begin{lem}\label{ergodyczne}
Assume that
$F:\mbR^{2N}\times\mbR^{2}\to\mbR$ is continuous in all variables and such that $y\mapsto F(\cdot,y)$, $y\in\bbR^2$ is twice differentiable in the $L^p(\nu_*^{y_0})$-sense for any $y_0\in\mbR^2$. 
 Then, the 
 process $F^{(\eps)}(t)$, obtained from $F$ by  formulas \eqref{ozn} and \eqref{tild},  satisfies
\begin{equation}
\label{020704}
\lim_{\ep\to0^+}\bbE\left| \sup_{t\in[0,T]}\left(\int_0^t
F^{(\eps)}\left(s\right)ds-\int_0^t\bar
F(x_\ep(s))ds\right)\right|=0,\quad T>0,
\end{equation}
where
$$
\bar F(y):=\int_{\mbR^{2N}} F(\fa;y)\nu_{*}^y(d\fa).
$$
\end{lem}
\begin{proof} By considering 
$F'({\frak a},y):=F({\frak a},y)-\bar F(y) $ we can and shall assume that
$F(\cdot,y)$ is of zero $\nu_*^y$-mean.
    Suppose that $\Theta(\cdot,y)$ is the $\nu_*^y$-mean
  zero  solution of the cell problem
\begin{equation}
\label{LF}
-{\cal L}\Theta(\cdot,y)=F(\cdot,y).
\end{equation}
 Thanks to the results of Section \ref{pochodna} we can apply  the
 It\^{o}-Krylov formula to the process
$\Theta^{(\ep)}\left(t\right)$, obtained from $\Theta$ by an
application of \eqref{ozn}.   Using \eqref{LF}  we get
\begin{align}\label{jedgwia}\begin{aligned}
& d
\left[\eps^2\Theta^{(\ep)}\left(t\right)\right]=\Bigg\{-F^{(\eps)}\left(t\right)+\ep
R_1^{(\eps)}(t)
+\eps^2R_2^{(\eps)}(t)\Bigg\}dt\\
&
+\ep\sum_{i}\Bigg\{(\sigma_i\sqrt{2\al_i})(x_\ep(t))
\left[\Theta^{(\ep)}_{a_i}\left( t\right)dw_{i,a}(t)+
\Theta^{(\ep)}_{b_i}\left( t\right)dw_{i,b}(t)\right]\Bigg\}
\end{aligned}\end{align}
for some standard, independent Brownian motions $dw_{i,a}(t),
dw_{i,b}(t)$ and
\begin{align*}
&
R_1^{(\eps)}(t):=\sum_{i,j} \left( F_{i,j}^{(\ep)}(t)a_{i,y_j}^{(\eps)}(t)+
 G_{i,j}^{(\ep)}(t)b_{i,y_j}^{(\eps)}(t)\right)+ V^{(\ep)}(t)\cdot\Theta^{(\ep)}_{y}\left( t\right),\\
&
R_2^{(\eps)}(t)
 :=\sum_{i,j}U^{(\ep)}_{j}\left(t\right)\bigg[\Theta^{(\eps)}_{a_i}\left( t\right) a_{i,y_j}^{(\eps)}(t) +\Theta^{(\ep)}_{b_i}\left( t\right) b_{i,y_j}^{(\eps)}(t) \bigg],
\end{align*}
with
\begin{align}\label{fij}\begin{aligned}
&F_{i,j}:=D^\perp_j\Theta+\mathfrak{w}_j\Theta_{a_i},\quad G_{i,j}:=\mathfrak{w}_j\Theta_{b_i},\quad
i=1,\ldots,N,\,j=1,2.
\end{aligned}\end{align}
In addition, Lemma \ref{SUPREMUM} implies that for any $T,r>0$ we have
$$
\lim_{\ep\to0}\ep\bbE\left\{
\sup_{t\in[0,T]}\bigg[\left|\Theta^{(\ep)}(t)\right|^r+\sum_i\left(\left|\Theta^{(\ep)}_{a_i}(t)\right|^r+\left|\Theta^{(\ep)}_{b_i}(t)\right|^r\right)+\left|\Theta_y^{(\ep)}(t)\right|^r\bigg]\right\}=0.
$$

Estimating as in \eqref{010704} -- \eqref{pommis} we conclude that all
terms appearing with factor $\eps$, or $\eps^2$ in \eqref{jedgwia}
vanish, as $\eps\to0$, hence \eqref{020704} follows. 
\end{proof}


\subsection{Tightness}\label{cias}
Next step in the proof of convergence in law of the processes
$\left(x_\ep(t)\right)_{t\ge0}$, as $\ep\to0$, is establishing the
tightness of their laws  over   $C\left([0,+\infty);\mbR^2\right)$.


From \eqref{ixy2}  we can write 
\begin{equation}
\label{ixy21}
x_{\ep,q}(t)={\cal B}^{(\ep)}_{q} \left(t\right)
 +\int_0^t{\cal M}_{q}^{(\ep)}\left(ds\right)+R^{(\eps)}_q(t)+\eps
{\cal N}^{(\ep)}_q(t),\quad q=1,2, t\ge0,
\end{equation}
where 
\begin{equation}
\label{020705}
{\cal B}^{(\ep)}_{q} \left(t\right):=\int_0^tB_q(x_{\ep}(s))ds
\end{equation}
and 
\begin{align}\label{pziv1}\begin{aligned}
& B_q(y) :=\sum_{i,j}\left\{\al_{i,y_j}(y)
\int_{\bbR^{2N}}\left(\Theta^{(1)}_{i,j}({\frak
    a},y)a_i+\Theta^{(2)}_{i,j}({\frak
    a},y)b_i\right)\right\}\nu_*^y(d{\frak a})\\
&+\int_{\bbR^{2N}}{\frak w}({\frak a})\cdot \chi_{q,y}({\frak a},y)
\nu_*^y(d{\frak a}).
\end{aligned}\end{align}
In addition,
$$
R^{(\eps)}_q(t):=\int_0^t\widetilde{\cal
    W}^{(\ep)}_{q} \left(s\right)ds-{\cal B}^{(\eps)}_q(t).
$$
The martingale term in \eqref{ixy21} is given by \eqref{mart-1}. Its
covariation  is given by 
\begin{align}\label{wachanieM}\begin{aligned}
&
\left\langle \int_0^\cdot{\cal M}_{q}^{(\ep)}\left(ds\right), \int_0^\cdot{\cal
  M}_{q'}^{(\ep)}\left(ds\right) \right\rangle_t=\int_0^tm_{q,q'}^{(\ep)}(s)ds,\quad\textrm{where}\\
&m_{q,q'}^{(\ep)}(s):=2\sum_{i=1}^N\left(\al_i\sigma_i^2\right)(x_\ep(s))\left[\chi_{q,a_i}^{(\eps)}\left(s\right) \chi_{q',a_i}^{(\eps)}\left(s\right) +\chi_{q,b_i}^{(\eps)}\left(s\right) \chi_{q',b_i}^{(\eps)}\left(s\right)\right]
\bigg\},\quad q,q'=1,2.
\end{aligned}\end{align}

An elementary application of Lemma \ref{ergodyczne} implies that
\begin{equation}
\label{040704}
\lim_{\ep\to0}\bbE\left[\sup_{t\in[0,T]}|R^{(\ep)}(t)|\right]=0,\quad T>0.
\end{equation}
Combining \eqref{040704} with \eqref{semi} we conclude the tightness
of $\left(x_\ep(t)\right)_{t\ge0}$ is equivalent with  the
tightness of the laws of  
\begin{equation}
\label{ixy22}
y_{\ep,q}(t)={\cal B}^{(\eps)}_q(t)
 +\int_0^t{\cal M}_{q}^{(\ep)}\left(ds\right),\quad q=1,2, t\ge0,
\end{equation}
over   $C\left([0,+\infty);\mbR^2\right)$. The latter follows,
provided that we show tightness of the family of processes
corresponding to the terms appearing on the right
hand side of  \eqref{ixy22}. Tightness of
$\left({\cal B}^{(\eps)}_q(t)\right)_{t\ge0}$
follows easily from an application of Theorem VI.5.17 of \cite{jacod} and the fact that coefficients $B_q$ are bounded.
 Concerning tightness of the laws of the processes corresponding to the
martingale part, according to Theorem  VI.4.13 of ibid., 
it is a consequence of the respective tightness of 
$$
\sum_{q=1}^2\int_0^tm_{q,q}^{(\ep)}(s)ds,\quad t\ge0,
$$
as $\eps\to0$. Using Lemma \ref{ergodyczne} we conclude that the
latter is a consequence of tightness of the laws of processes
\begin{equation}
\label{030705}
\sum_{q=1}^2\int_0^tA_{q,q}(x_\eps(s))ds,\quad t\ge0,
\end{equation}
with $A_{q,q'}(y)$ given by  
\begin{equation}
\label{Aqq}
A_{q,q'}(y):=2\sum_{i=1}^N\left(\al_i\sigma_i^2\right)(y)\int_{\bbR^{2N}}\left[\left(\chi_{q,a_i} \chi_{q',a_i}\right)\left({\frak
    a},y\right) +\left(\chi_{q,b_i} \chi_{q',b_i}\right)\left({\frak
    a},y\right)\right]\nu_*^y(d {\frak
    a}).
\end{equation}
The latter
follows from yet another application of  Theorem VI.5.17 of \cite{jacod}. 

\subsection{Identification of the limit}\label{identyf}
We have already mentioned that limiting laws of
 $\left(x_\eps(t)\right)_{t\ge0}$ and $\left(y_\eps(t)\right)_{t\ge0}$
 coincide, as 
\begin{equation}
\label{010705}
\lim_{\eps\to0}\bbE\left[\sup_{t\in[0,T]}|x_\eps(t)-y_\eps(t)|\right]=0,\quad T>0.
\end{equation}
Using the It\^{o} formula we conclude from \eqref{ixy22} that  for any function
 $f\in C^2\left(\mbR^2\right)$,
\begin{align*}
&M_\eps(t;f):=f(y_\eps(t))-f(y_\eps(0))-\sum_{q=1}^2\int_0^t\partial_{y_q}f(y_\ep(s))B_{q} \left(x_\eps(s)\right)ds\nonumber\\
&-\frac12\sum_{q,q'=1}^2\int_0^t\partial_{y_q
y_{q'}}^2f(y_\ep(s))A_{q,q'}(x_\eps(s))ds,\quad t\ge0
\end{align*}
is a martingale, where $B_q(y)$, $A_{q,q'}(y)$ are given by
\eqref{pziv1} and \eqref{Aqq}. 
Applying the results of Section \ref{pochodna} we conclude that the
coefficients are continuous. Thanks to \eqref{010705}
we conclude that any limiting law of $(x_\ep(t))_{t\ge0}$, as
$\eps\to0$ has to solve the martingale problem corresponding to the
operator
$$
\overline{\mathcal{L}}f(y)=\sum_{q=1}^2B_q(y)\partial_{y_q}f(y)+\frac12\sum_{q,q'=1}^2A_{q,q'}(y)\partial^2_{y_{q,q'}}f(y),
$$
which by virtue of  Theorem 7.2.1 of \cite{Stroock-Varadhan}   is well
posed. So the limiting law  is uniquely determined. This ends the
proof of Theorem \ref{glowne}.\qed

\subsection{The  case of an arbitrary dimension}

\label{sec5.5}

Because of the notational convenience we have proved Theorem
\ref{glowne}  only in the two dimensional case. The proof in an
arbitrary dimension is virtually the same. Here we discuss briefly how
to modify the respective formulas in order to obtain the expressions
for the drift and diffusivity coefficients $B_q$, $A_{q,q'}$,
$q,q'=1,\ldots,d$ in the  $d-$dimensional situation. Recall that then the
modes of the velocity field given by \eqref{050705} are indexed by the
elements of the set $Z$ defined in \eqref{Z}.

The correctors $\chi_q(\cdot;y)$ corresponding to \eqref{011702} are
given by   the zero $\nu_*^y$-mean solutuions of equations
\begin{equation}\label{011702j}
 -{\cal L}_y\chi_q(\cdot;y)=\mathfrak{w}_q,
\end{equation}
with ${\frak w}=({\frak w}_1,\ldots, {\frak w}_d)$ given by 
$$
 {\frak w}_q({\frak a}):=\sum_{i,l}k_{i,l}b_{i,l,q},\quad q=1,\ldots,d,
$$  
where ${\cal L}_y$ is the generator of the diffusion $\tilde\fa(t;y):=\left(\tilde a_\jj(t;y),\tilde
  b_\jj(t;y)\right)_{\jj\in Z}$, with
\begin{align}
&\tilde a_\jj(t;y):=a_\jj(t;y)\cos(k_i\cdot
z(t,y))+b_\jj(t;y)\sin(k_i\cdot z(t,y)),\nonumber\\
&\tilde b_\jj(t;y):=-a_\jj(t;y)\sin(k_i\cdot
z(t,y))+b_\jj(t;y)\cos(k_i\cdot z(t,y)).\label{tyldab1}
\end{align}
Here $z(t,y)$ is  the solution of \eqref{dynPom}.
The generator takes the form \eqref{gen} with 
\begin{equation}\label{obrot-j}
D_q F:=\sum_\jj k_{i,q}\mathcal{R}_\jj F,\quad q=1,\ldots,d
\end{equation}
and 
\begin{equation}\label{Rjj}
\mathcal{R}_\jj F:=\left( b_\jj\partial_{a_\jj}-
a_\jj\partial_{b_\jj}\right)F,\qquad \jj=(i,l,m)\in Z,\,F\in C^1(\mbR^{2S})
\end{equation}
 (recall that $S$ is the cardinality of $Z$).
Finally, we solve the systems
\begin{align}\label{Theduz12}\begin{aligned}
& \left[{\cal L}_y- \al_\jj(y) \right]\Theta^{(1,q)}_{\jj,j}(\cdot;y)-(k_i \cdot
\mathfrak{w}) \Theta^{(2,q)}_{\jj,j}(\cdot;y)=F_{\jj,j}^q(\cdot;y),\\
& \left[{\cal L}_y- \al_\jj(y) \right]\Theta^{(2,q)}_{\jj,j}(\cdot;y)+(k_i \cdot
\mathfrak{w}) \Theta^{(1,q)}_{\jj,j} (\cdot;y)=G_{\jj,j}^q(\cdot;y),
\end{aligned}\end{align}
with
\begin{align*}
&
F_{\jj,j}^q({\frak a};y):=\delta_{l,j}\left(\delta_{m,q}+D_m\chi_q({\frak
  a})\right) +{\frak w}_j({\frak a})\chi_{q,a_\jj}({\frak a},y),\\
&
G_{\jj,j}^q({\frak a};y):={\frak w}_j({\frak a})\chi_{q,b_\jj}({\frak a},y),\quad\jj=(i,l,m)\in Z, \quad\,j,q=1,\ldots,d.
\end{align*}
Then,
the formulas for the
 coefficients of the limiting diffusions are as follows
\begin{align}\label{wspD}\begin{aligned}
&B_q(y):=\int_{\mbR^{2S}}\mathfrak{w}(\fa)\cdot\chi_{q,y}\left(\fa,y\right) \nu_{*,y}(d\fa)+\sum_{\jj,j}\big\{ \al_{\jj,y_j}(y)\int_{\bbR^{2S}}
\left(\Theta^{(1,q)}_{\jj,j}\left(\fa,y\right)a_\jj+\Theta^{(2,q)}_{\jj,j}\left(\fa,y\right)
  b_\jj\right)\big\}\nu_{*,y}(d\fa),\\
&\\
&A_{q,q'}(y):=2\sum_\jj\left(\al_\jj\sigma_\jj^2\right)(y)\int_{\mbR^{2N}}\Big[\chi_{q,a_\jj}\left(\fa,y\right)\chi_{q',a_\jj}(\fa,y)
+\chi_{q,b_\jj}\left(\fa,y\right)\chi_{q',b_\jj}(\fa,y)\Big]\nu_{*}^y(d\fa)
\end{aligned}\end{align}
for $q,q'=1,\ldots,d$.

\section{Regularity of the corrector}
\label{pochodna}

\subsection{Corrector problem}

\label{CP}

The present section is concerned with regularity of  solutions
$\Xi:\bbR^{2N+2}\to \bbC$ of the  equation
\begin{equation}\label{PT-91}
\mathcal{L}_y\Xi(\fa;y)+c(\fa;y)\Xi(\fa;y)=f(\fa;y),\quad (\fa,y)\in \bbR^{2N+2}.
\end{equation}
A complex valued function
$f:\bbR^{2N+2}\to \bbC$ is assumed to satisfy
\begin{align}\label{azl2}\begin{aligned}
\sum_{m=0}^2
\sup_{y\in\mbR^2}\|\nabla_y^mf(\cdot,y)\|_{L^{p}(\nu_{*}^y)}<+\infty,\quad
\textrm{for each } p\in [1,+\infty).
\end{aligned}\end{align}

Concerning the function $c$ we will consider two cases: either 1)
$c\equiv0$ and then we assume
$$
\int_{\bbR^{2N}}f(\fa;y)\nu_*^y(d\fa)=0,\quad y\in\bbR^2,
$$
 or
2) $c(\fa,y)=-\alpha_i(y)+\textrm{i}q(\fa;y)$ (i - the imaginary unit), where $i\in\{1,\dots,N\}$ is fixed
while $q$ is a real valued polynomial of the second degree in the variable
$\fa$. The coefficients of the polynomial $q(\fa;y)$
are assumed to be $C^2_b(\mbR^2)$ regular functions of the variable $y$.
The operator $\mathcal{L}_y$ was defined in
\eqref{gen}. 

According to the results of Section \ref{korekotsek} and formula
\eqref{Feynamnkac}, under the above assumptions there exists a unique solution
to  \eqref{PT-91}, which in the case 1) satisfies 
$$
\int_{\bbR^{2N}}\Xi(\fa;y)\nu_*^y(d\fa)=0,\quad y\in\bbR^2.
$$
Using the Feynman-Kac formula we obtain that
\begin{align}\label{Feynamnkac1}\begin{aligned}
&\Xi(\fa;y)=\int_0^{+\infty}\bbE\Bigg[\exp\left\{\int_0^tc(\tilde\fa^{\fa}(s;y),y)ds\right\}f(\tilde\fa^{\fa}(t;y);y)\Bigg]dt,\quad \fa\in\bbR^{2N}.
\end{aligned}\end{align}
Here $\tilde\fa^\fa(t,y)=\left(\tilde a_i^\fa(t,y),\tilde
  b_i^\fa(t,y)\right)$ is
the solution of \eqref{dyf2} satisfying $\tilde\fa^\fa(0,y)=\fa$. Thanks
to \eqref{pspectral}
we conclude from \eqref{Feynamnkac1} that for each $p\in(1,+\infty)$
we have
\begin{equation}
\label{010709}
\left\|\Xi(\cdot;y)\right\|_{L^p(\nu_{*}^y)}\preceq \left\|f(\cdot;y)\right\|_{L^p(\nu_{*}^y)}\quad y\in\bbR^2.
\end{equation}

\subsection{$L^p$ regularity of  the corrector in the ${\frak a}$-variable}

\label{Malavain}

Our first result concerns the $L^p$ regularity of the solutions of the equation
\begin{equation}\label{PT-91a}
-\mathcal{L}_y\Xi(\fa;y)=\ff(\fa;y),\quad (\fa,y)\in \bbR^{2N+2},
\end{equation}
in the $\fa$ variable.
Here $\ff:\bbR^{2N+2}\to \mathbb C$  is such that $\ff(\cdot,y) \in
L^q(\nu_{*}^y)$ for some $q>1$.
Thanks to \eqref{inv1} we conclude that
 \begin{equation}
\label{050707}
\int_{\bbR^{2N}}\ff(\fa;y)\nu_*^y(d\fa)=0,\quad y\in\bbR^2,
\end{equation}
 is a necessary condition for
its solvability.
Using \eqref{Feynamnkac1} we can write
\begin{equation}\label{pomoc}
\Xi(\fa,y)=\int_0^{+\infty}P_t^{y}\ff(\fa;y)dt=\int_0^{+\infty}\bbE\left[\ff(\tilde\fa^{\frak a}(t,y),y)\right]dt.
\end{equation}
\begin{twr}\label{TwProf}
Assume  that  $\Xi$ is given by \eqref{pomoc} and $q\in(1,+\infty)$.
Then, for any $p\in[1,q)$ there exists $C>0$ (independent of $y$) such
that
\begin{equation}\label{Proff1}
\left\|\nabla_\fa\Xi(\cdot,y)\right\|_{L^{p}(\nu_{*}^y)}+\left\|\nabla_\fa^2\Xi(\cdot,y)\right\|_{L^{p}(\nu_{*}^y)}\le
C\left\|\ff(\cdot,y)\right\|_{L^q(\nu_{*}^y)},\quad
y\in\bbR^2
\end{equation}
for any $\ff:\bbR^{2N+2}\to \mathbb C$  such that $\ff(\cdot,y) \in
L^q(\nu_{*}^y)$ for all $y\in\bbR^2$.
\end{twr}
The proof of the above result is presented in Section \ref{sec6.2} but
first we apply it to conclude the following.
\begin{corollary}
\label{cor010709}
Under the assumptions of Section $\ref{CP}$ for any $1\le p<q<+\infty$
there exists $C>0$ such that
\eqref{PT-91} satisfies 
\begin{equation}\label{Proff}
\left\|\Xi(\cdot,y)\right\|_{W^{2,p}(\nu_{*}^y)}\le
C\left\|f(\cdot,y)\right\|_{L^{q}(\nu_{*}^y)}\quad
\mbox{ for all }y\in\bbR^2,\,f\in L^q(\nu_*^y).
\end{equation}
\end{corollary}
\proof
We let
\begin{equation}
\label{fff}\ff(\fa;y)=-c(\fa;y)\Xi(\fa;y)+f(\fa;y).
\end{equation}
Thanks to \eqref{azl2} and \eqref{010709} we conlcude that
$\ff(\cdot,y) \in
L^q(\nu_{*}^y)$  for any $q\in[1,+\infty)$ and $y\in\bbR^2$. Then \eqref{Proff} is a
direct  consequence of \eqref{Proff1} and \eqref{010709}.
\qed

\subsection{Proof of Theorem \ref{TwProf}}

\label{sec6.2}

We show first that for each $p\in[1,q)$ there exists $C>0$ such that
\begin{equation}
\label{010706}
\|\nabla_{\frak a}\Xi(\cdot,y)\|_{L^p(\nu_*^y)}\le
C\left\|\ff(\cdot,y)\right\|_{L^q(\nu_{*}^y)},\quad y\in\bbR^2.
\end{equation}
In the proof we focus only on estimating  the $L^p(\nu_{*}^y)$ norm of 
 $\Xi_{b_i}(\cdot,y)$, $i=1,\dots,N$. The argument for  $\Xi_{a_i}(\cdot,y)$
is similar.
 In addition, we assume that $\ff$ is differentiable in the $\fa$ variable. The constant $C>0$ in estimate \eqref{010706} turns out not to depend on $\nabla_{\frak a}\ff$ so we can relax this assumption by approximation.

From \eqref{pomoc} we obtain
\begin{equation}\label{mamy}
\Xi_{b_j}(\fa,y)=\int_0^{+\infty}v_j(t)dt,
\end{equation}
where 
\begin{eqnarray}
\label{010706d}
&&v_j(t):=\partial_{b_j}P_t^y\ff(\fa)=\bbE\left[(\nabla_\fa \ff)( \tilde
\fa^{\frak a}(t,y))\cdot D_{b_j}\tilde
\fa^{\frak a}(t,y)\right]\\
&&
=\sum_{i}\bbE\left[\ff_{a_i}(\tilde \fa^{\frak a}(t,y),y)
\xi_{i,j}(t)+\ff_{b_i}(\tilde \fa^{\frak a}(t,y),y)\eta_{i,j}(t)\right]\nonumber
\end{eqnarray}
and $D_{b_j}\tilde
\fa^{\frak a}(t,y)=\left(\xi_{i,j}(t),
  \eta_{i,j}(t)\right)_{i=1,\ldots,N}$ is  the Fr\'{e}chet derivative
of the stochastic flow $\fa\mapsto \tilde \fa^{\frak a}(t,y)$, with
$\xi_{i,j}(t):=\tilde a^{\frak a}_{i,b_j}(t,y)$ and
$\eta_{i,j}(t):=\tilde b^{\frak a}_{i,b_j}(t,y)$.

Differentiating \eqref{dyf2} with respect to the initial condition we
conclude that
\begin{align}
\label{070707}
\begin{aligned}
&\frac{d\xi_{i,j}}{dt}=-\alpha_{i}\xi_{i,j}+\sum_{i'}\delta\left(k_{i},
k_{i'}\right) \left(\tilde b^{\frak a}_{i'}\eta_{i,j}+
\tilde b_{i}^{\frak a}\eta_{i',j}\right),\\
&\frac{d\eta_{i,j}}{dt}=-\alpha_{i}\eta_{i,j}-\sum_{i'}\delta\left(k_{i},
k_{i'}\right)\left(\tilde a^{\frak a}_{i}\eta_{i',j}+
\tilde b_{i'}^{\frak a}\xi_{i,j}\right),\\
& \xi_{i,j}(0):=0,\ \ \eta_{i,j}(0):=\delta_{i,j}, \
i,j=1,\dots,N,
\end{aligned}\end{align}
where $\delta_{i,j}$ is the Kronecker symbol, i.e.  $\delta_{i,i}=1$ and $\delta_{i,j}=0$, if $i\not=j$.

We shall prove that
\begin{equation}
\label{080108}
\|v_j(t)\|_{L^p(\nu_*^y)}\preceq 
\left\{
\begin{array}{ll}
e^{-\gamma(q)t}\|\ff(\cdot,y)\|_{L^q(\nu_*^y)},& t\ge 1,\\
&\\
\dfrac{1}{t^{1/2}}\|\ff(\cdot,y)\|_{L^q(\nu_*^y)}&t\in(0,1),
\end{array}
\right.
\end{equation}
where $\gamma(q)$ is the same as in \eqref{pspectral}. Estimate
\eqref{010706}  then follows from the above bound and formula \eqref{mamy}.

Consider first the case when $t\ge1$. For given  ${\frak
  g}(t)=\left(g_i^{(a)}(t),g_i^{(b)}(t)\right)_{i=1,\ldots,N}$ with
$g_i^{(a)}, g_i^{(b)} \in L^2[0,+\infty)$, $i=1,\ldots,N$ we let
{\begin{align}
\label{030707x}
&\|\!|{\frak g}\|\!|_{r,t}:=
\sup_{y\in\bbR^2}\left\{\sum_{i}\Bigg\{\left\{\int_{\bbR^{2N}}\nu_*^y(d\fa)\bbE\left[\frac1t\int_0^t|g_i^{(a)}{(s,y,\fa)}|^2ds\right]^{r/2}\right\}^{1/r}\right.\nonumber\\
&\left.+\sum_{i}\left\{\int_{\bbR^{2N}}\nu_*^y(d\fa)\bbE\left[\frac1t\int_0^t|g_i^{(b)}{(s,y,\fa)}|^2ds\right]^{r/2}\right\}^{1/r}\Bigg\}\right\}.
\end{align}
We shall write $\|\!|{\frak g}\|\!|_{r}:=\|\!|{\frak g}\|\!|_{r,1}$.}
Let also
${\frak h}(t)=(h_i^{(a)}(t),h_i^{(b)}(t))_{i=1,\ldots,N}$, where
\begin{equation}\label{ggg}
h_i^{(a)}(t):=\int_0^tg_i^{(a)}(s)ds,\quad h_i^{(b)}(t):=\int_0^tg_i^{(b)}(s)ds,\quad t\ge0.
\end{equation}
Treating the solution $\tilde \fa^{\frak a}(t,y; w)$ of \eqref{dyf2} as the functional of the Wiener process $w(t)=(w_{i,a}(t),w_{i,b}(t))_{i=1,\ldots,N}$,
we define the Malliavin derivative of $\tilde \fa^{\frak a}(t,y;w)$ in the direction ${\frak h}$
\begin{align*}
&\mathcal{D}_{\frak h}\tilde \fa^{\frak a}(t,y;w):=\lim_{\ep\to0}\frac{1}{\ep}\left\{\tilde \fa^{\frak a}(t,y;w+\ep
{\frak h})-\tilde \fa^{\frak a}(t,y;w)\right\},
\end{align*}
 where the limit above is understood in the $L^2$ sense.

 
 Denoting
$\zeta_i(t):=\mathcal{D}_{\frak h} \tilde a^{\frak a}_i(t,y;w)$ and 
$\theta_i(t):=\mathcal{D}_{\frak h} \tilde b^{\frak a}_i(t,y;w)$,
$i=1,\ldots,N$ the components of the Malliavin derivative, we can see, from \eqref{dyf2},
that they satisfy
\begin{align}\label{zetatheta}\begin{aligned}
&\frac{d\zeta_{i}}{dt}=-\alpha_{i}\zeta_{i}+\sum_{i'} \delta\left(k_{i},
k_{i'}\right)\left(\tilde b_{i}^{\frak a}\theta_{i'}+
\tilde b_{i'}^{\frak a}\theta_{i}\right)+\sqrt{2\alpha_i}\sigma_ig_i^{(a)},\\
&\frac{d\theta_{i}}{dt}=-\alpha_{i}\theta_{i}-\sum_{i'}\delta\left(k_{i},
k_{i'}\right)\left(\tilde a_{i}^{\frak a}\theta_{i'}+\tilde b_{i'}^{\frak a}\zeta_{i}\right)+\sqrt{2\alpha_i}\sigma_ig_i^{(b)},\\
& \zeta_{i}(0):=0,\ \ \theta_i(0):=0, \ i=1,\dots,N.
\end{aligned}\end{align}
The difference of the Fr\'echet and Malliavin derivatives
$$
\Gamma(t,y) := D_{b_j}\tilde {\frak  a}^{\frak
               a}(t,y)-\mathcal{D}_{\frak h}\tilde {\frak  a}^{\frak
              a}(t,y)=
(\Upsilon_{i,j}(t),\Theta_{i,j}(t))_{i=1,\ldots,N}
$$
solves the following system of equations
\begin{align}\label{xitheta}\begin{aligned}
&\frac{d\Upsilon_{i,j}}{dt}=-\alpha_{i}\Upsilon_{i,j}+\sum_{i'}\delta\left(k_{i},
k_{i'}\right)\left(\tilde b_{i}^{\frak a}\Theta_{i',j}+
\tilde b_{i'}^{\frak a}\Theta_{i,j}\right)-\sqrt{2\alpha_i}\sigma_ig_i^{(a)},\\
&\frac{d\Theta_{i,j}}{dt}=-\alpha_{i}\Theta_{i}-\sum_{i'}\delta\left(k_{i},
k_{i'}\right)\left(\tilde a_{i}^{\frak a}\Theta_{i',j}+
\tilde b_{i'}^{\frak a}\Upsilon_{i,j}\right)-\sqrt{2\alpha_i}\sigma_ig_i^{(b)},\\
& \Upsilon_{i,j}(0):=0,\ \ \Theta_{i,j}(0):=\delta_{i,j}, \
i=1,\dots,N.
\end{aligned}\end{align}

Therefore (see \eqref{010706d}), from the chain rule for the Malliavin
derivative, see Proposition 1.2.3, p. 28 of \cite{Nualart}, we obtain
\begin{align}
\label{020707}v_j(t)=\tilde v_j(t)+\bbE\left[\mathcal{D}_g\ff(\tilde \fa^{\frak
  a}(t,y),y)\right],
\end{align}
where
$$
\tilde v_j(t):=\sum_{i}\bbE\Big[\partial_{a_i}\ff(\tilde \fa^{\frak a}(t,y),y)
\Upsilon_{i,j}(t)+\partial_{b_i}\ff(\tilde \fa^{\frak a}(t,y),y)\Theta_{i,j}(t)\Big].
$$
Integrating by parts the second term on the right hand side of
\eqref{020707} (see Lemma 1.2.1 p. 25,
of \cite{Nualart}) we
conclude that
\begin{align}\label{wzor1}\begin{aligned}
&v_j(t)=\tilde v_j(t)
+\sum_{i}\bbE\left[\ff(\tilde \fa^{\frak a}(t,y),y)\left(\int_0^tg_i^{(a)}(s)dw_{i,a}(s)+\int_0^tg_i^{(b)}dw_{i,b}\right)\right].
\end{aligned}\end{align}
We shall look for the control ${\frak g}(t,y,{\frak a})=\left(g_i^{(a)}(t,y, {\frak a}),g_i^{(b)}(t,y, {\frak a})\right)_{i=1,\ldots,N}$, which satisfies the following conditions:
\begin{itemize}
\item[i)] it is adapted with respect to the natural filtration of
  $\left(w(t)\right)_{t\ge0}$,
\item[ii)] the respective $\Gamma(t,y)\equiv0$ and ${\frak
    g}(t,y, {\frak a})\equiv0$ for $t\ge1$,
  $({\frak a} ,y)\in\bbR^{2N+2}$,
\item[iii)] we have (cf \eqref{030707x})
{\begin{equation}
\label{030707}
\|\!|{\frak g}\|\!|_{r}=\|\!|{\frak g}\|\!|_{r,1}<+\infty.
\end{equation}}
\end{itemize}

{
Then, thanks to ii), we conclude that $\tilde v_j(t)\equiv0$.}
Using formula \eqref{wzor1} and the Markov property of
$\left(\fa^{\frak a}(t,y)\right)_{t\ge0}$ we can write
\begin{align*}
&v_j(t)=\sum_i\bbE\left[P_{t-1}^y\ff(\tilde \fa^\fa(1,y))\left(\int_0^1g_i^{(a)}(s,y)dw_{i,a}(s)+\int_0^1g_i^{(b)}(s,y)dw_{i,b}(s)\right)\right],\quad
\mbox{  for  }t\ge1.
\end{align*}
Applying H\"{o}lder's inequality with $1/q+1/r=1/p$ we obtain
\begin{align*}
&\|v_j(t)\|_{L^p(\nu_*^y)}\leq\|P_{t-1}^y\ff(\cdot,y)\|_{L^q(\nu_*)}\\
&
\times \sum_{i}\left\{\left\{\int_{\bbR^{2N}}\nu_*^y(d\fa)\bbE\left|\int_0^1g_i^{(a)}(s,y,\fa)dw_{i,a}(s)\right|^r\right\}^{1/r}
+\left\{\int_{\bbR^{2N}}\nu_*^y(d\fa)\bbE\left|\int_0^1g_i^{(b)}(s,y,\fa)dw_{i,b}(s)\right|^r\right\}^{1/r}\right\}.
\end{align*}
Thanks to \eqref{050707} we can apply spectral gap estimate
\eqref{spectralgap2}. This and the Burkohlder-Davis-Gundy inequality
imply the following bound
\begin{align}
\label{010708-18}
&\|v_j(t)\|_{L^p(\nu_*^y)}\preceq 
e^{-\gamma(q)(t-1)}\|\ff(\cdot,y)\|_{L^q(\nu_*^y)}\|\!| {\frak g}\|\!|_r,\quad t\ge1,\,y\in\bbR^2
\end{align}
for  $\gamma(q)>0$ as in \eqref{pspectral}.

{When, on the other hand $t\in(0,1)$ we represent
  $v_j(t)=\partial_{b_j}P_t^y\ff(\fa)$ using the Bismut-Elworthy-Li formula,
  see e.g. formula (3.3.24), p. 75 of \cite{elworthy},
\begin{equation}
\label{bel}
\partial_{b_j}P_t^y\ff(\fa)=\frac1t\bbE\left[\ff( \tilde
\fa^{\frak a}(t,y))\int_0^t\Sigma^{-1}D_{b_j}\tilde
\fa^{\frak a}(s,y)\cdot dw(s)\right].
\end{equation}
The matrix $\Sigma$ is diagonal and given by formula
\begin{equation}
\label{Sigma}
\Sigma:=\textrm{diag}[\sqrt{2\alpha_1}\sigma_1,\dots,\sqrt{2\alpha_N}\sigma_N,\sqrt{2\alpha_1}\sigma_1,\dots,\sqrt{2\alpha_N}\sigma_N].
\end{equation}
Hence, after using the H\"older inequality and lower bounds \eqref{gamma},
we get
\begin{align}
\label{020708-18}
&\|v_j(t)\|_{L^p(\nu_*^y)}\leq \frac{1}{\si_*\gamma_0 t}
\|\ff(\cdot,y)\|_{L^q(\nu_*^y)}\\
&
\times \sum_{i}\left\{\left\{\int_{\bbR^{2N}}\nu_*^y(d\fa)\bbE\left|\int_0^t\xi_{i,j}(s)dw_{i,a}(s)\right|^r\right\}^{1/r}
+\left\{\int_{\bbR^{2N}}\nu_*^y(d\fa)\bbE\left|\int_0^t\eta_{i,j}(s)dw_{i,b}(s)\right|^r\right\}^{1/r}\right\}.\nonumber
\end{align}
Applying subsequently
Burkholder-Davis-Gundy and Jensen inequalities, we obtain 
\begin{align}
\label{030708-18}
&\|v_j(t)\|_{L^p(\nu_*^y)}\preceq \frac{1}{ t^{1/2}}
\|\ff(\cdot,y)\|_{L^q(\nu_*^y)}\|\! D_{b_j}\tilde
\fa^{\frak a}(\cdot,y)\|\!|_{r,t}
\le \frac{1}{ t^{1/2}}\|\ff(\cdot,y)\|_{L^q(\nu_*^y)} \\
&
\times \sum_{i}\left\{\left\{\frac1t\int_0^t\int_{\bbR^{2N}}\nu_*^y(d\fa)\bbE\left|\xi_{i,j}(s) \right|^r ds\right\}^{1/r}+
\left\{\frac1t\int_0^t\int_{\bbR^{2N}}\nu_*^y(d\fa)\bbE\left|\eta_{i,j}(s) \right|^r ds\right\}^{1/r}\right\}.\nonumber
\end{align}
Thanks to \eqref{050708} we conclude that
\begin{equation}
\label{070708}
\sup_{t\in(0,1),\,y\in\bbR^2}\sum_{i}\left\{\left\{\frac1t\int_0^t\int_{\bbR^{2N}}\nu_*^y(d\fa)\bbE\left|\xi_{i,j}(s) \right|^r ds\right\}^{1/r}+
\left\{\frac1t\int_0^t\int_{\bbR^{2N}}\nu_*^y(d\fa)\bbE\left|\eta_{i,j}(s)
  \right|^r ds\right\}^{1/r}\right\}<+\infty.
\end{equation}
Thus,
\begin{equation}
\label{060709a}
\|v_j(t)\|_{L^p(\nu_*^y)}\preceq \frac{1}{
  t^{1/2}}\|\ff(\cdot,y)\|_{L^q(\nu_*^y)},\quad t\in(0,1),\quad y\in\bbR^2.
\end{equation}}
From estimates \eqref{010708-18} and \eqref{060709a} we conclude
\eqref{080108}, which ends the proof of \eqref{010706}, 
 provided
we can find a control ${\frak g}$ which satisfies conditions i) - iii)
and show estimate \eqref{070708}. We shall deal with these issues in
Section \ref{sec6.4}.
The above argument can be conducted in the case of $a_j$ variables as well, so we conclude \eqref{010706}.

Using \eqref{gen} we infer that for each $p\in[1,+\infty)$
\begin{equation}
\label{010706a}
\|\mathbb L^y\Xi(\cdot,y)\|_{L^p(\nu_*^y)}\le \|{\frak f}(\cdot,y)\|_{L^p(\nu_*^y)}+\|\mathfrak{w}\cdot D\Xi(\cdot,y)\|_{L^p(\nu_*^y)}.
\end{equation}
From the definition of the operator $D$, see \eqref{obrot}, and
H\"older inequality, applied to the second term on the right hand side
of \eqref{010706a}, we obtain
\begin{equation}
\label{010706b}
\|\mathfrak{w}\cdot D\Xi(\cdot,y)\|_{L^p(\nu_*^y)}\le \|\Phi\|_{L^{r'}(\nu_*^y)}\|\nabla_{\fa}\Xi(\cdot,y)\|_{L^{q'}(\nu_*^y)},
\end{equation}
where $\Phi$ is some second degree polynomial in $\fa$ with constant coefficients. We have assumed that 
$q'\in(p,q)$ and $r'$ are such that $1/q'+1/r'=1/p$. Using the
already proved estimate  \eqref{010706} to bound the norm of the
gradient on the right hand side of \eqref{010706b}, we conclude that
$$
\|\mathfrak{w}\cdot D\Xi(\cdot,y)\|_{L^p(\nu_*^y)}\preceq
\|\ff(\cdot,y)\|_{L^{q}(\nu_*^y)},\quad y\in\bbR^2.
$$
Thus,
\begin{equation}
\label{010706c}
\|\mathbb L^y\Xi(\cdot,y)\|_{L^p(\nu_*^y)}\preceq \|\ff(\cdot,y)\|_{L^q(\nu_*^y)},\quad y\in\bbR^2.
\end{equation}
Since,  see Theorem 1.5.1 of \cite{Nualart}, p. 72, for each $p\in(1,+\infty)$ we have 
$$
\|\nabla^2_{\frak a}\Xi(\cdot,y)\|_{L^p(\nu_*^y)}\preceq \|\mathbb L^y\Xi(\cdot,y)\|_{L^p(\nu_*^y)},\quad y\in\bbR^2.
$$
This estimate allows us to conclude the proof of Theorem \ref{TwProf}.\qed

\subsection{Construction of a control ${\frak g}$ and proof of (\ref{070708})}

\label{sec6.4}

Denote by $C(t,s)=[C_{i,i'}(t,s)]_{i,i'=1,\ldots,2N}$ the fundamental matrix of the system
\eqref{070707}.
It is a $2N\times2N$-matrix, which is the solution of the equation
\begin{equation}\label{eq16}
\frac{d}{dt}C(t,s)=A(t)C(t,s),\quad C(s,s)=I_{2N},\quad t,s\ge0,
\end{equation}
where $I_{2N}$ is the  identity $2N\times2N$-matrix and $A(t)=[A_{i,i'}(t)]_{i,i'=1,\ldots,2N}$, where
\begin{align}
\label{090707}
&A_{i,i'}(t):=-\alpha_i\delta_{i,i'},\qquad
  A_{i+N,i'+N}(t):=-\left[\alpha_i\delta_{i,i'}+\delta\left(k_i,k_{i'}\right) \tilde a_{i}^{\frak a}(t,y)\right],\\
&A_{i,i'+N}(t):=\delta\left(k_i,k_{i'}\right) \left(\tilde b_{i}^{\frak a}(t,y)+\delta_{i,i'}\sum_{i''}\delta\left(k_i,k_{i''}\right)\tilde b_{i''}^\fa(t,y)\right),\nonumber\\
&A_{i+N,i'}(t):=\delta_{i,i'}\sum_{i''}\delta\left(k_i,k_{i''}\right)\tilde b_{i''}^\fa(t,y),\quad 1\leq i,i'\leq N.\nonumber
\end{align}
We  have 
$$
C(u,t)C(t,s)=C(u,s),\quad u,t,s\in\mbR.
$$
System \eqref{xitheta} can be rewritten as follows
\begin{equation}\label{eq15}
\frac{d\Phi}{dt}=A(t)\Phi-\Sigma{\frak g}(t,y),\quad \Phi(0)=E,
\end{equation}
where $\Phi$ is the $2N\times N$-dimensional matrix such that
$$
\Phi_{i,j}:=\Upsilon_{i,j},\quad 1\le i,j\le N,\quad
\Phi_{i,j}:=\Theta_{i,j},\quad N+1\leq i\le 2N,\ 1\le j\le N,
$$
and $E$ is a block vector, such that
$
E^T:=\left[ 0_N,
I_{N}\right],
$ where $0_N$, $I_N$ are  the $N\times N$ null matrix and identity
matrix, respectively.
Here the diagonal matrix $\Sigma$ is defined in \eqref{Sigma}
and
${\frak
  g}:=[g_1^{(a)},\ldots,g_N^{(a)},g_1^{(b)},\ldots,g_N^{(b)}]^T.$
The solution of equation \eqref{eq15} can be expressed by the
fundamental solution as follows
\begin{equation}
\label{080707}
\Phi(t)=-\int_0^tC(t,s)\Sigma{\frak g}(s)ds+C(t,0)E,\quad t\ge0,
\end{equation}
which in turn implies that $\Gamma(t,y)\equiv 0$, so condition ii) is satisfied.

Let
\begin{equation}
\label{100707}
{\frak g}(t,y,{\frak a}):=\Sigma^{-1}C(t,0)E, \quad t\in[0,1],
\end{equation}
and ${\frak g}(t,y,{\frak a})\equiv 0$ for $t\ge 1$. The process  is adapted with
respect to the  natural filtration of $\left(w_t\right)_{t\ge0}$,
satisfying therefore condition 
i). Additionally, we have
$$
-\int_0^1C(1,s)\Sigma{\frak g}(s,y,\fa)ds+C(1,0)E=-\int_0^1 C(1,0) Eds+C(1,0)E=0.
$$
Thus,  $\Phi(t)\equiv0$ for $t\ge1$, which in turn implies that
$\Gamma(t,y)\equiv0$, $t\ge1$. Condition ii) is therefore fulfilled.

It remains to be checked that $\left({\frak g}(t,y,\fa)\right)_{t\ge0}$,
constructed above,
satisfies the estimate \eqref{030707}.
From \eqref{eq16} we conclude that
$$
\|C(t,0)\|\leq 1+\int_0^t{\frak A}(s)\|C(s,0)\|ds ,\quad
t\geq0,
$$
where 
$$
\|C(t,s)\|:=\max_{i,i'}|C_{i,i'}(t,s)|\quad\mbox{and}\quad {\frak
  A}(t):=\sum_{i,i'}|A_{i,i'}(t)|
$$
Using first Gronwall and then  the Jensen  inequality for the
integral in $ds$, we obtain
\begin{equation}
\label{080708}
{\sup_{t\in[0,1]}\|C(t,0)\|}\leq\exp\left\{\int_0^1{\frak A}(s)ds\right\}\le \int_0^1\exp\{{\frak A}(s)\}ds.
\end{equation}
It is clear from \eqref{090707} that
\begin{equation}
\label{110707}
{\frak A}(t)\preceq 1+\sum_{i}(|a_i^{\frak a}(t,y)|+|b_i^{\frak a}(t,y)|),
\end{equation}
where $\left( a_i^{\frak a}(t,y),b_i^{\frak a}(t,y)\right)_{t\ge0}$ is
the solution of \eqref{Ornstein}, with  $\left( a_i^{\frak
    a}(0,y),b_i^{\frak a}(0,y)\right)_{i=1,\ldots,N}={\frak a}$. 

Recall that $\min_i\{\sigma_i\}\ge \sigma_*>0$ and
$\min_i\{\alpha_i\}\ge \gamma_0>0$. Therefore, from \eqref{100707},
\eqref{110707} and gaussianity of  $\left( a_i^{\frak
    a}(t,y),b_i^{\frak a}(t,y)\right)_{t\ge0}$  we obtain that for
each $r\ge1$ we have
\begin{align*}
&\|\!|{\frak
  g}\|\!|_r\preceq\sup_{y\in\bbR^2}\left\{\int_0^1ds\int_{\bbR^{2N}}\nu_*^y(d\fa)\bbE\exp\{r{\frak
  A}(s)\}\right\}^{1/r}<+\infty.
\end{align*}

{Since
$
D_{b_j}\tilde
\fa^{\frak a}(t,y)=C(t,0)E_j,
$
where $E_j$ is the $j$-th column vector of the matrix $E$,
from \eqref{080708} we conclude that
\begin{equation}
\label{050708}
 \sup_{t\in[0,1],\,y\in\bbR^2}\int_{\bbR^{2N}}\nu_*^y(d\fa)\bbE\|D_{b_j}\tilde
\fa^{\frak a}(t,y)\|^r\preceq \sup_{y\in\bbR^2}\int_0^1ds\int_{\bbR^{2N}}\nu_*^y(d\fa)\bbE\exp\{r{\frak
  A}(s)\}<+\infty.
\end{equation}}

\subsection{Regularity of the corrector in the $y$-variable}


We start with the following simple lemma.
\begin{lem}\label{lemP} For any $1\le p<q<+\infty$ there exist $C,r>0$ such that
  \begin{equation}
\label{020709}
\sup_{|y-y_0|\le r}\|g\|_{L^p(\nu_{*}^y) }\le C
\|g\|_{L^q(\nu_{*}^{y_0})},\quad y_0\in\bbR^2. 
\end{equation}
\end{lem}
\begin{proof}
Observe that
\begin{equation}\label{Bam}
\|g\|^p_{L^p(\nu_{*}^y)}=\int_{\mbR^{2N}}|g|^p\rho d\nu_{*}^{y_0}=\int_{\mbR^{2N}}|g|^p(\rho -1)d\nu_{*}^{y_0}+\int_{\mbR^{2N}}|g|^pd\nu_{*}^{y_0},
\end{equation}
where
$$
\rho(y):=\prod_{i=1}^N\left\{\dfrac{\sigma_i^2(y_0)}{\sigma_i^2(y)}\exp\left\{-\dfrac{(a_i^2+b_i^2)(\sigma_i^2(y)-\sigma_i^2(y_0))}{2\sigma_i^2(y_0)\sigma_i^2(y)}\right\}\right\},\quad
y\in\bbR^2.
$$
Using an elementary inequality
$
|e^x-1|\leq  e^{2|x|}
$
we obtain that
\begin{align}
&\int|g|^p|\rho-1|d\nu_{*}^{y_0}\leq \int_{\mbR^{2N}}|g|^p
\exp\left\{\sum_{i=1}^N\dfrac{(a_i^2+b_i^2)|\sigma_i^2(y)-\sigma_i^2(y_0)|}{\sigma_i^2(y_0)\sigma_i^2(y)}\right\}d\nu_{*}^{y_0}\label{dalszyciag}.
\end{align}
From uniform continuity of the function $\sigma_i^2$ (because
$\sigma_i\in C^2_b(\mbR^d)$) for any $\mu>0$ there exists $\delta>0$
such that $|\sigma_i^2(y)-\sigma_i^2(y_0)|<\mu$, provided that
$|y-y_0|<\delta$. From the H\"older inequality with $1/q+1/q'=1/p$ and the
lower bound \eqref{gamma} on $\si$,  we obtain that the expression
 \eqref{dalszyciag} is less than or equal
 \begin{align*}
&\|g\|_{L^q(\nu_{*}^{y_0})} \Bigg(\int_{\mbR^{2N}}\exp\left\{\frac{q'\mu}{\sigma_*^{4q'}}
\sum_{i=1}^N\left(a_i^2+b_i^2\right)\right\}
d\nu_*^{y_0}\Bigg)^{1/2}.
\end{align*}
One can choose $q'\mu$  sufficiently small so that the
second factor  is finite, which ends the proof of the lemma.
\end{proof}

\bigskip


The main result of this section is the following.
\begin{twr}\label{lab}
Suppose that $\Xi:\bbR^{2N+2}\to\mathbb C$ is the solution of
\eqref{PT-91a}. Then, under the assumptions made in Section
$\ref{CP}$ for any $p\in[1,+\infty)$
there exists $r>0$ such that 
for any $y_0\in\bbR^2$ 
we have $\Xi(\cdot,y)\in W^{2,p}(\nu_*^{y_0})$, provided that
$|y-y_0|<r$ and
\begin{equation}
\label{040709}
\lim_{y\to y_0}\|\Xi(\cdot,y)-\Xi(\cdot,y_0)\|_{ W^{2,p}(\nu_*^{y_0})}=0.
\end{equation}
For each $p\in[1,+\infty)$ the derivatives $\nabla_y^{m}\Xi(\cdot,y)$, $m=1,2$
exist in the $W^{p,2}(\nu_{*}^y)$-sense for each  $y\in\bbR^2$. In
addition, 
 \begin{equation}
 \label{extra-3}
 \sum_{m=0}^2\sup_{y\in\mbR^2}\|\nabla_y^{m}\Xi(\cdot,y)\|_{W^{p,2}(\nu_{*}^y)}<+\infty.
\end{equation}
\end{twr}
\begin{proof}
To simplify the notation we shall assume that parameter
$y\in\mbR$. Given a function $f:\bbR^{2N+2}\to\mathbb C$  we let
$\delta f(\cdot,y_0):=
f(\cdot,y)-f(\cdot,y_0)$. From \eqref{PT-91} we can write 
\begin{equation}\label{ilor}
\mathcal{L}_{y_0}\delta\Xi(\fa;y)+c(\fa,y_0)
\delta\Xi(\fa;y)=-\delta\mathcal{L}_{y_0}\Xi(\fa;y)-\delta c(\fa,y_0) \Xi(\fa;y)+\delta f(\fa,y_0) ,
\end{equation}
where $\delta\mathcal{L}_{y_0}$ is the differential operator obtained
from $\mathcal{L}_{y_0}$ by taking the corresponding differences of
the coefficients.
Using Lemma \ref{lemP} and Corollary \ref{cor010709} we conclude that
for each $q\in[1,+\infty)$ 
the $L^q(\nu_*^{y_0})$-norm of the right hand side tends to $0$, as
$y\to y_0$.  Equality \eqref{040709} is then a consequence of
\eqref{Proff}.





To prove the existence of the derivative $\partial_y \Xi(\fa;y)$ denote
by 
$${\cal D}_h f(\cdot,y_0):=
\frac1h[f(\cdot,y_0+h)-f(\cdot,y_0)]
$$
for a given  function $f:\bbR^{2N+2}\to\mathbb C$ and  $h\not=0$.
We show that
\begin{equation}
\label{060709}
\lim_{h\to0}\|R_h\Xi(\cdot,y_0) \|_{ W^{2,p}(\nu_*^{y_0})}=0,
\end{equation}  
where 
$$
R_h\Xi(\cdot,y_0) :={\cal D}_h \Xi(\cdot,y_0)-\partial_y\Xi(\cdot,y_0)
$$
and
$\nabla_y\Xi(\cdot,y_0) $ is the solution of 
$$
\mathcal{L}_{y_0}\partial_y\Xi(\fa;y_0)+c(\fa,y_0)
 \partial_y\Xi(\fa;y_0)=-\mathcal{L}_{y_0}'\Xi(\fa;y_0)- \partial_y c(\fa,y_0) \Xi(\fa;y_0)+\partial_y f(\fa,y_0) .
$$
Here $ \mathcal{L}_{y_0}'$ is the differential operator obtained from
$ \mathcal{L}_{y_0}$ by differentiating in $y$ its coefficients.
We have
\begin{align*}
&\mathcal{L}_{y_0} R_h\Xi(\cdot,y_0) +c(\fa,y_0)
 R_h\Xi(\cdot,y_0) =\mathcal{L}_{y_0}'\Xi(\fa;y_0)- {\cal D}_h
  \mathcal{L}_{y_0}\Xi(\fa;y_0+h) \\
&
+\partial_y c(\fa,y_0) \Xi(\fa;y_0)- {\cal D}_hc(\fa,y_0) \Xi(\fa;y_0+h) +\partial_y f(\fa,y_0) - {\cal D}_hf(\fa,y_0).
\end{align*}
Here $ {\cal D}_h\mathcal{L}_{y_0}$ is the differential operator obtained from
$ \mathcal{L}_{y_0}$ by taking the respective quotients of its coefficients.
Using again estimate \eqref{Proff} we conclude that
 $$
 \lim_{h\to0}\|R_h\Xi(\cdot,y_0)\|_{ W^{2,p}(\nu_*^{y_0})}=0\quad\mbox{ for any }y_0\in\bbR.
 $$
 The proof of the existence of the second derivative is analogous.
 
{The fact that $\sup_y \|\Xi(\cdot,y)\|_{W^{p,2}(\nu_*^y)}<+\infty$ for any $p\in[1,+\infty)$ follows directly by an application of Corollary \ref{cor010709}. One can see from \eqref{ilor} that  $\partial_y\Xi(\cdot,y)$  satisfies  equation of the form \eqref{PT-91} with the right hand side that 
belongs to  $L^p(\nu_*^y)$ for any $p\in[1,+\infty)$. Another application of Corollary \ref{cor010709} yields that  $\sup_y \|\partial_y\Xi(\cdot,y)\|_{W^{p,2}(\nu_*^y)}<+\infty$. A similar argument can be also made for  $\partial_y^2\Xi(\cdot,y)$ and estimate  \eqref{extra-3} follows.}
 \end{proof}

%
%

\subsection{$L^\infty$ estimates of the corrector}\label{lnies}

Given a differentiable function
$f:\bbR^{2N}\to\mathbb C$ and $R>0$ we define the norm
$$
\|f\|_{1,\infty}^{(R)}:=\sup_{|\fa|\le R}\left(|f(\fa)|+|\nabla_\fa f(\fa)|\right).
$$
The main purpose of this section is the proof of the following result.
\begin{Propozycja}\label{Prop}
Under the assumptions made in Section \ref{CP} for any $C_*>0$ there exists $C>0$ such that \begin{equation}\label{biala}
\sup_{y\in\mbR^2}\sum_{m=0}^2\|\nabla_y^m\Xi(\cdot;y)\|_{1,\infty}^{(R)}\leq
Ce^{C_*R^2}\quad\mbox{for all }R>0.
\end{equation}
\end{Propozycja}
\begin{proof}
Let $C_*>0$ be arbitrary. We choose  $p\in[1,+\infty)$ to be specified further later on. Note that  (cf \eqref{gamma})
\begin{equation}\label{jazda}
\|\Xi(\cdot;y)\|_{W^{2,p}(B_R)}\preceq \|\Xi(\cdot;y)\|_{W^{2,p}(\nu_{*}^y)}e^{R^2/(p\sigma_*^2)},\quad \mbox{for all }
R\geq0,\, y\in\bbR^2.
\end{equation}
%
%
Due to the Sobolev embdedding,  see e.g. Theorem 7.10,   p. 155 of 
\cite{Gilbar}, space $W^{2,p}(B_R)$ can be embeded into $C^1(B_R)$, provided that $p>d$. In consequence there exists $C>0$ such that
\begin{equation}
\label{070709}
\|\Xi(\cdot;y)\|_{1,\infty}^{(R)}\leq C(R+1)^{2-d/p}\|\Xi(\cdot;y)\|_{W^{2,p}(B_R)} \quad\mbox { for all } y\in\bbR^2,\,R>0.
\end{equation}
From \eqref{Proff}, \eqref{070709} and \eqref{jazda} we conclude that  there exists $C>0$ such that
\begin{equation}
\label{070709a}
\|\Xi(\cdot;y)\|_{1,\infty}^{(R)}\leq C(R+1)^{2-d/p}e^{R^2/(p\sigma_*^2)} \quad\mbox { for all } y\in\bbR^2,\,R>0.
\end{equation}
Choosing  $p>2/(C_*\sigma_*^2)$, we conclude that for some $R_0$ 
$$
(R+1)^{2-d/p}e^{R^2/(p\sigma_*^2)}\le e^{C_*R^2} \quad\mbox { for all }R\ge R_0.
$$
Increasing suitably the constant $C>0$, if necessary, and recalling \eqref{010709} we conclude 
that
$$
\sup_{y\in\mbR^2}\|\Xi(\cdot;y)\|_{1,\infty}^{(R)}\leq
Ce^{C_*R^2}\quad\mbox { for all }R.
$$
The proof of the bounds on the respective norms of  $\nabla_y\Xi(\cdot;y)$ and  $\nabla_y^2\Xi(\cdot;y)$ can be done analogously,
thus
\eqref{biala} follows.
\end{proof}
%
%
%
%
%
%
%

\section{Bounds on the moments of  suprema of  some Gaussian processes}

\label{sec12.1}

Let  $I$ be an arbitrary set. We say that a field $\left(\bA(t,z)\right)_{
(t,z)\in [0,+\infty)\times I}$ is stationary in the $t$-variable  if for any $h\ge0$ the laws  of the field  and that of
$\left(\bA(t+h,z)\right)_{ (t,z)\in [0,+\infty)\times I}$ are identical.
\begin{Propozycja}\label{Gaussowskie}
Let $I$ be a compact metric space and $N$ some natural number. Assume that
$\left(\bA(t,z)\right)_{
(t,z)\in [0,+\infty)\times I}$ is  an $\mbR^{N}$-valued, Gaussian and  stationary in the $t$-variable  random field  with continuous realizations. Then for any $\gamma\in
(0,1),T>1,$ there exist $C,C'>0$ such that
\begin{equation}
\label{080709}
\bbE\left\{\sup_{t\in[0,T],z\in
I}\exp\left\{C\left|\bA\left(\frac{t}{\ep},z\right)\right|^2\right\}\right\}\leq
C'\ep^{-\gamma},\quad \ep\in(0,1].
\end{equation}
\end{Propozycja}
\begin{proof}
Consider the Banach space $E:=C([0,1]\times I)$ with the standard supremum norm $\|\cdot\|_{E}$. From the Borell-Fernique-Talagrand  theorem, see  e.g. Theorem 2.6, p. 37 of \cite{daza},  there exists $\lambda_0>0$ such that for $0<\lambda<\lambda_0$ we have
\begin{equation}
\label{BFT}
\bbE\exp\left\{\lambda\left\|\bA(\cdot)\right\|^2_E\right\}<+\infty.
\end{equation}
Choose any $C\in(0,\gamma\lambda_0).$
Consider the random variables
$$
{\cal S}_i:=\exp\left\{C\sup_{(t,z)\in[i,i+1]\times
I}\left|\bA\left(t,z\right)\right|^2\right\},\quad
i=0,\dots,T(\eps):=\left[\frac{T}{\ep}\right]+1,
$$
where $[x]$ denotes the largest integer that is less than, or equal to $x\in\bbR$. Thanks to the stationarity in $t$ of the field $\bA$ the laws of all these random variables are identical.
Let $\la:=C/\ga$.  From  \eqref{BFT} we  have 
\begin{equation}
\label{100709}
\bbE
{\cal S}_0^{1/\gamma}\le \bbE\exp\left\{\lambda\left\|\bA(\cdot)\right\|^2_E\right\}<+\infty. 
\end{equation}
The left hand side of \eqref{080709} can be estimated by
from above by
\begin{align*}
&\bbE\left[\max_{i=0,\dots T(\eps)}{\cal S}_i\right]\leq\left\{\bbE\max_{i=0,\dots,
T(\eps)}{\cal S}_i^{1/\gamma}\right\}^{\gamma}\leq\left(\bbE
\sum_{i=0}^{T(\eps)}{\cal S}_i^{1/\gamma}\right)^{\gamma}\leq
(2T)^{\gamma}\ep^{-\gamma}\left(\bbE {\cal S}_0^{1/\gamma}\right)^{\gamma}.
\end{align*}
and \eqref{080709} follows.
\end{proof}

\bigskip

\begin{corollary}\label{uw1} Suppose that $\left(\fa(t,y)\right)_{(t,y)\in\bbR^{3}}$ is the field defined in \eqref{roz}. 
For any $\gamma\in(0,1)$ and $ T>0$ there exist $C,C'>0$
such that
\begin{equation}\label{gausnier}
\sum_{m=0}^2\bbE\left\{\sup_{t\in[0,T]}\sup_{y\in\mbR^d}\exp\left\{C\left|\nabla_y^m\fa\left(\frac{t}{\ep^2},y\right)\right|^2\right\}\right\}\leq
C'\ep^{-\gamma},\quad \ep\in(0,1].
\end{equation}
\end{corollary}
\begin{proof} Indeed let (cf \eqref{gamma} and \eqref{roz})
\begin{align}\label{duzeA}\begin{aligned}
&A_i(t,z,z')=z'\int_{-\infty}^te^{-z(t-s)}dw_{i,a}(s),\\
&B_i(t,z,z')=z'\int_{-\infty}^te^{-z(t-s)}dw_{i,b}(s),\quad
(z,z')\in I,
\end{aligned}\end{align} where
$I:=[\gamma_0,\gamma_0^{-1}]\times\left[\sqrt{2\ga_0}/\sigma_*,\sqrt{2}/(\sqrt{\ga_0}\sigma_*)\right]$, the Brownian motions $w_{i,a}(s),
w_{i,b}(s)$ are as in $\eqref{Ornstein}$ for $i=1,\dots,N$.
The conclusion of the corollary concerning the term of \eqref{gausnier} corresponding to $m=0$ follows from an application of  Proposition \ref{Gaussowskie} to the field $\bA(t,z,z')=:(A_i(t,z,z'),B_i(t,z,z'))_{i=1,\dots,N}$.
Since the functions $\alpha_i,\sigma_i$  appearing in the definitions of the respective fields $a_i(t,y)$ and $b_i(t,y)$ are of 
$C^2_b(\mbR^2)$ class of regularity  we can repeat the above argument for the terms corresponding to $m=1,2$ as well. 
\end{proof}



\section{Application  to estimates of moments of suprema related to the corrector along the tracer path}

\label{sec8}

Recall that  $\Xi^{(\ep)}(t)=\tilde\Xi(t,\bar x_\eps(t),x_\eps(t))$, where  $\tilde\Xi(t,x,y)=\Xi(\tau_x\fa(t),y)$ and $\bar x_\eps(t):=\eps^{-1}x_\eps(t)$.
The processes $\nabla_{\fa}\Xi^{(\ep)}(t)$ and $\nabla_y\Xi^{(\ep)}(t)$ are formed similarly, using $\nabla_{\fa}\Xi$ and $\nabla_{y}\Xi$ instead of $\Xi$.
\begin{lem}\label{SUPREMUM}
Under the assumptions made in Section $\ref{CP}$ for any $T,r>0$ we have
\begin{align}\label{armed}\begin{aligned}
&\lim_{\ep\to0}\ep\bbE\left\{
\sup_{t\in[0,T]}\bigg[\left|\Xi^{(\ep)}(t)\right|^r+\left|\nabla_{\fa}\Xi^{(\ep)}(t)\right|^r+\left|\nabla_y\Xi^{(\ep)}(t)\right|^r\bigg]\right\}=0.
\end{aligned}\end{align}
\end{lem}
\begin{proof}  We conduct the proof for the process $\Xi^{(\ep)}(t)$. For the other processes appearing in \eqref{armed} the argument  is similar. From  Proposition
\ref{Prop} we know that for any  $r>0$ and constant $C_*$  there exists constant $C>0$ such that for all
$\ep\in(0,1),\ t\geq0$
\begin{align}\label{cytowane}\begin{aligned}
&\left|\Xi^{(\ep)}\left(t\right)\right|^r\leq
C\exp\left\{C_*\left|\tau_{\bar x_\ep(t)}\fa\left(\frac{t}{\ep^2};x_\ep(t)\right)\right|^2\right\}\leq
C\sup_{y\in\mbR^d}\exp\left\{C_*\left|\fa\left(\frac{t}{\ep^2};y\right)\right|^2\right\}.
\end{aligned}\end{align} 
Thanks to Corollary \ref{uw1} we can choose $C_*$ in such a way that  for some $\gamma\in(0,1)$ and $C'>0$ the supremum of the right hand side over $T\in[0,T]$ of \eqref{cytowane} can be estimated by $C'\eps^{-\ga}.$ Thus,
$$
\lim_{\ep\to0}\ep\bbE\left[
\sup_{t\in[0,T]}\left|\Xi^{(\ep)}(t)\right|^r\right]=0.
$$
\end{proof}


\begin{lem}\label{ar}
For any $ t,r\geq0$ we have
\begin{equation}\label{armed2}
\limsup_{\ep\to0}\bbE\left\{
\left|\Xi^{(\ep)}(t)\right|^r+\left|\nabla_\fa\Xi^{(\ep)}(t)\right|^r+\left|\Xi_y^{(\ep)}(t)\right|^r\right\}<+\infty.
\end{equation}
\end{lem}
\begin{proof}
Again, we present the argument only for $\Xi^{(\ep)}(t)$. 
Using estimate \eqref{cytowane} we can see that the conclusion of the lemma follows, provided we can show that for some $C_*>0$
\begin{equation}
\label{010710}
\limsup_{\ep\to0}\bbE\left\{\sup_{y}\exp\left\{C_*\left|\fa\left(\frac{t}{\ep^2};y\right)\right|^2\right\}\right\}<+\infty.
\end{equation}
From time stationarity of the Ornstein-Uhlenbeck processes we conclude that
\begin{align}
&\bbE\left\{\sup_{y}\exp\left\{C_*\left|\fa\left(\frac{t}{\ep^2};y\right)\right|^2\right\}\right\}=\bbE\left\{\exp\left\{C_*\sup_{y}\left|\fa\left(0;y\right)\right|^2\right\}\right\}.\nonumber
\end{align}
Recall that from \eqref{roz} we have
$$
a_i(0;y)=\sqrt{2\alpha_i(y)}\sigma_i(y)\int_{-\infty}^0e^{\alpha_i(y)s}dw_{i,a}(s),\quad
i=1,\dots,N,\ y\in\mbR^2.
$$
Therefore (cf \eqref{gamma})
$$
\sup_y|a_i(0;y)|\leq
\frac1\sigma_*\sup_{z\in[\gamma_0,\gamma_0^{-1}]}|A_i(z)|,
$$
where
$$
A_i(z):=\sqrt{2z}\int_{-\infty}^0e^{z s}dw_{i,a}(s),\quad z\in[\gamma_0,\gamma_0^{-1}].
$$
Field $A_i(\cdot)$ is Gaussian with covariance function
$$
R(z,z')=\frac{2\sqrt{zz'}}{z+z'},\quad z,z'\in[\gamma_0,\gamma_0^{-1}].
$$
Using results of \cite{Adler}, (see Theorem 5.2, p. 120) we know that there exists constants
 $C,K,\ \lambda_0>0$ such that
$$
\bbP\left(\sup_{z\in[\gamma_0,1/\gamma_0]}|A_i(z)|\geq\lambda\right)\leq
C\lambda^K\exp\left\{-\frac{\lambda^2}{2R_*}\right\},\quad
i=1,\dots,N,\ \lambda>\lambda_0,
$$
where
$
R_*:=\sup\left\{R(z,z'),z,z'\in[\gamma_0,1/\gamma_0]\right\}.
$
Therefore for a sufficiently small $C_*>0$ formula \eqref{010710} holds. 
\end{proof}

\end{document}